\documentclass[10pt]{amsart}
\usepackage[T1]{fontenc}
\usepackage{lmodern}
\usepackage{amsmath,amssymb,mathtools}
\usepackage{mathrsfs}
\usepackage{microtype}
\usepackage{booktabs}
\usepackage[hyphens]{url}
\usepackage[hidelinks]{hyperref}
\usepackage{xcolor}
\usepackage[all]{xy}
\usepackage{tikz}
\usetikzlibrary{arrows.meta}

\hypersetup{
 colorlinks,
 linkcolor={teal},
 citecolor={teal},
 urlcolor={teal}
}

\calclayout

\newtheorem{theorem}{Theorem}[section]
\newtheorem{proposition}[theorem]{Proposition}
\newtheorem{lemma}[theorem]{Lemma}
\newtheorem{corollary}[theorem]{Corollary}
\theoremstyle{definition}
\newtheorem{definition}[theorem]{Definition}
\newtheorem{example}[theorem]{Example}
\theoremstyle{remark}

\theoremstyle{plain}
\newtheorem*{theoremA}{Theorem A}
\newtheorem*{theoremB}{Theorem B}
\newtheorem*{theoremC}{Theorem C}
\numberwithin{equation}{section}

\DeclareMathOperator{\Mag}{Mag}
\DeclareMathOperator{\MH}{MH}
\DeclareMathOperator{\MC}{MC}
\DeclareMathOperator{\OS}{OS}

\DeclareMathOperator{\wt}{wt}
\DeclareMathOperator{\Fun}{Fun}
\DeclareMathOperator{\gr}{gr}
\DeclareMathOperator{\supp}{supp}
\DeclareMathOperator{\cl}{cl}
\DeclareMathOperator{\Flag}{Flag}
\newcommand{\OM}{\mathcal M}
\newcommand{\cT}{\mathcal T}
\newcommand{\cL}{\mathcal L}
\newcommand{\cF}{\mathcal F}
\newcommand{\Ftwo}{\mathbb F_2}
\newcommand{\eps}{\varepsilon}
\newcommand{\Z}{\mathbb Z}
\newcommand{\Q}{\mathbb Q}
\newcommand{\R}{\mathbb R}
\newcommand{\N}{\mathbb N}
\newcommand{\one}{\mathbf 1}

\title[Magnitude and motivic zeta functions of matroids]
{Magnitude and motivic zeta functions of matroids}
\author[J. Koizumi]{Junnosuke Koizumi}
\address{RIKEN iTHEMS, Wako, Saitama 351-0198, Japan}
\email{junnosuke.koizumi@riken.jp}
\date{}
\subjclass[2020]{Primary 05B35, 52C40; Secondary 05C12, 52C35, 55N35}
\keywords{Matroid, oriented matroid, magnitude, magnitude homology,
motivic zeta function, Varchenko--Gelfand filtration}

\begin{document}
\begin{abstract}
We prove that the magnitude of the tope graph of a simple oriented matroid
is the specialization at $x=-1$ of the motivic zeta function of its
underlying matroid.
This refines the Las Vergnas--Zaslavsky theorem and defines magnitude
for arbitrary matroids.
For simple orientable matroids, we compute the order of the pole at
$q=-1$ from chains of flats and identify it with the
Varchenko--Gelfand degree of the tope parity function.
Comparing pole orders, we construct a rank-six real arrangement
$\mathcal A$ such that $\Mag(\mathcal A;-t)$ has infinitely many
negative coefficients, disproving Koizumi--Liu's eventual sign
alternation conjecture.
We then construct a canonical multiplicative Varchenko--Gelfand
filtration on the mod-$2$ magnitude cohomology of a simple oriented
matroid and show that its graded dimensions recover the full motivic
zeta function.
\end{abstract}
\maketitle
\enlargethispage{4pt}
\setcounter{tocdepth}{1}
\tableofcontents

\section{Introduction}\label{sec:introduction}

Zaslavsky's theorem relates the geometry of a real hyperplane arrangement
to the combinatorics of its intersection lattice: the number of chambers
of a central arrangement is $(-1)^r\chi_M(-1)$, where $M$ is the corresponding matroid
and $r=r(M)$ \cite{Zaslavsky}.
Las Vergnas \cite{LasVergnas} showed that the same formula counts topes of a loopless oriented matroid.
The chambers are the vertices of the \emph{tope graph}, in which two chambers are adjacent when separated by
a single hyperplane, and the distance between two chambers counts the
hyperplanes separating them.
It is therefore natural to ask whether the chamber-counting formula
extends to an invariant that records these distances.

Magnitude is a candidate for such an invariant.
Defined by Leinster \cite{LeinsterMetric,Leinster} for metric spaces, and in particular for graphs, magnitude is a rational-function invariant with properties resembling cardinality.
In the expansion of the magnitude $\Mag(G;q)$ of a graph $G$ at $q=0$,
the constant term gives the number of vertices, and the linear term gives the number of edges.
Koizumi--Liu \cite{KL} initiated the study of magnitude and magnitude homology of tope graphs of real hyperplane arrangements.
They gave a recursive formula for computing magnitude from localizations and showed, among other things, that its poles are roots of unity other than $1$.

In the present paper, we prove that the \emph{motivic zeta function} of Jensen--Kutler--Usatine
\cite{JKU} gives the corresponding refinement of the characteristic polynomial.
For a function on a smooth variety, its motivic zeta function is a motivic analogue of the $p$-adic Igusa zeta function, defined by Denef--Loeser \cite{DL} using motivic integration.
Kutler--Usatine \cite[Theorem~1.10]{KU} showed that, for hyperplane arrangements, this function admits a combinatorial expression.
Jensen--Kutler--Usatine \cite{JKU} extended it to a rational function
$Z_M(x,T)$ defined for an arbitrary matroid $M$.
We use its combinatorial definition, recalled in
Section~\ref{sec:preliminaries}, throughout.
Its constant term is
\[
 Z_M(x,0)=x^{-r(M)}\chi_M(x).
\]
Our first main theorem refines the Las Vergnas--Zaslavsky theorem in terms of magnitude and the motivic zeta function:

\begin{theoremA}[Theorem~\ref{thm:zaslavsky}]
Let $\OM$ be a finite simple oriented matroid, with underlying matroid
$M$ and tope graph $G$.
Then
\begin{equation}\label{eq:intro-zaslavsky}
 \Mag(G;q)=Z_M(-1,q).
\end{equation}
\end{theoremA}

At $q=0$, \eqref{eq:intro-zaslavsky} recovers the Las Vergnas--Zaslavsky
formula.
In particular, the magnitude of the tope graph depends only on the
underlying matroid, even though the graph depends on the orientation.
The proof first extends Koizumi--Liu's face decomposition formula to
all simple oriented matroids.
Grouping faces by their zero sets then gives the same recursion as the
specialized motivic zeta function.

This identity leads to a definition for every matroid $M$:
\[
 \Mag(M;q):=Z_M(-1,q).
\]
The resulting invariant is rational, multiplicative under direct sums,
and valuative.
For loopless matroids it satisfies
\[
 \Mag(M;q^{-1})=q^{|E|}\Mag(M;q),\qquad \Mag(M;1)=1,
\]
and its poles are roots of unity other than $1$.
These properties extend those of arrangement magnitude in \cite{KL}.
Orientability places further restrictions: for a nonempty simple
orientable matroid, magnitude is greater than $1$ throughout $(-1,1)$
and has no real zeros.
By contrast, we show that $\Mag(\operatorname{PG}(8,2);9/10)<0$.
Therefore, the positivity of magnitude is a phenomenon specific to orientable matroids.

Magnitude (co)homology \cite{HW,Hepworth} refines magnitude by replacing
each coefficient with a graded group.
It has a cohomological degree $k$ and a length degree $\ell$; the alternating
sum of dimensions over $k$ at fixed length $\ell$ is the coefficient
of $q^\ell$ in magnitude.
In bidegree $(0,0)$, magnitude cohomology is the algebra of functions
on topes, the Varchenko--Gelfand algebra \cite{VG}.
Its Heaviside functions indicate the positive side of each coordinate.
Filtering by polynomial degree in these functions recovers the
characteristic polynomial from the dimensions of graded pieces.
We extend this filtration to all of magnitude cohomology and show that the dimensions of its graded pieces recover the full motivic zeta function.

\begin{theoremB}[Theorems~\ref{thm:zeta-filtration} and
\ref{thm:associated-graded}]
Let $G$ be the tope graph of a finite simple oriented matroid $\OM$ with
underlying matroid $M$.
Then its mod-$2$ magnitude cohomology has a canonical multiplicative filtration
$F$ extending the Varchenko--Gelfand filtration in degree $(0,0)$.
If
\[
 \mathscr H_{\OM}(s,t,T)=\sum_{p,k,\ell\geq0}
 \dim_{\Ftwo}\gr_p^F\MH^{k,\ell}(G;\Ftwo)\,s^pt^kT^\ell,
\]
then
\begin{equation}\label{eq:intro-filtered-zeta}
 Z_M(x,T)=\mathscr H_{\OM}(-x^{-1},x^{-1},T).
\end{equation}
The associated graded algebra is commutative and is determined by $M$.
\end{theoremB}

Over $\Ftwo$, each rank-one summand of the (co)homology gives
a unique nonzero element.
These classes make the construction independent of choices of
generators; the distinction from integral coefficients is discussed
in Section~\ref{sec:filtration}.
Setting $T=0$ in Theorem~B yields a description of the characteristic polynomial in terms of the Varchenko--Gelfand filtration.
Setting $x=-1$ corresponds to taking the alternating sum of the ranks of magnitude homology, thereby recovering Theorem~A.

\vskip\baselineskip

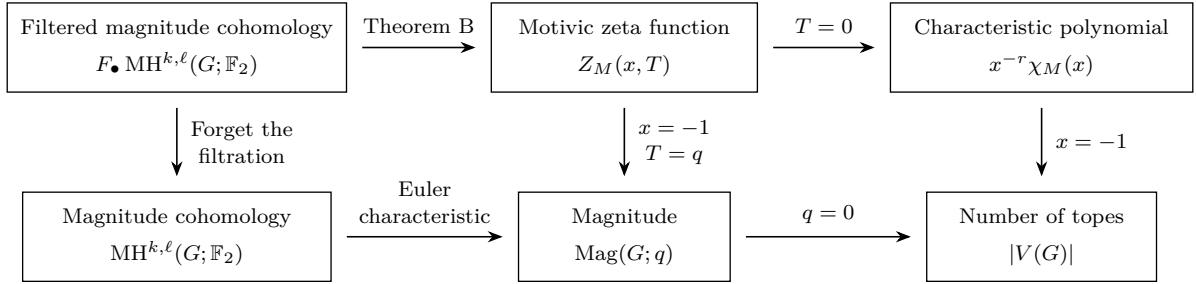
\begin{center}
\begin{tikzpicture}[
  x=\linewidth,
  invariant/.style={draw, line width=.5pt, align=center,
    font=\footnotesize, inner xsep=5pt, inner ysep=6pt,
    minimum height=10mm},
  arrow/.style={-{Stealth[length=2mm]}, line width=.6pt,
    shorten <=2mm, shorten >=2mm},
  explanation/.style={font=\footnotesize, align=center, inner sep=5pt}
]
  \node[invariant, minimum width=.265\linewidth] (filtered) at (.145,0)
    {Filtered magnitude cohomology\\[4pt]
     $F_\bullet\MH^{k,\ell}(G;\Ftwo)$};
  \node[invariant, minimum width=.22\linewidth] (zeta) at (.515,0)
    {Motivic zeta function\\[4pt] $Z_M(x,T)$};
  \node[invariant, minimum width=.25\linewidth] (characteristic) at (.86,0)
    {Characteristic polynomial\\[4pt] $x^{-r}\chi_M(x)$};

  \node[invariant, minimum width=.255\linewidth] (cohomology) at (.145,-2.5)
    {Magnitude cohomology\\[4pt] $\MH^{k,\ell}(G;\Ftwo)$};
  \node[invariant, minimum width=.175\linewidth] (magnitude) at (.515,-2.5)
    {Magnitude\\[4pt] $\Mag(G;q)$};
  \node[invariant, minimum width=.19\linewidth] (topes) at (.86,-2.5)
    {Number of topes\\[4pt] $|V(G)|$};

  \draw[arrow] (filtered.east) --
    node[explanation, above] {Theorem B} (zeta.west);
  \draw[arrow] (zeta.east) --
    node[explanation, above] {$T=0$} (characteristic.west);
  \draw[arrow] (filtered.south) --
    node[explanation, right] {Forget the\\filtration} (cohomology.north);
  \draw[arrow] (zeta.south) --
    node[explanation, right] {$x=-1$\\$T=q$} (magnitude.north);
  \draw[arrow] (characteristic.south) --
    node[explanation, right] {$x=-1$} (topes.north);
  \draw[arrow] (cohomology.east) --
    node[explanation, above] {Euler\\characteristic} (magnitude.west);
  \draw[arrow] (magnitude.east) --
    node[explanation, above] {$q=0$} (topes.west);
\end{tikzpicture}
\end{center}

\vskip\baselineskip

The top-left arrow takes graded dimensions and substitutes
$(s,t)=(-x^{-1},x^{-1})$ as in Theorem~B.
The middle and right vertical arrows are Theorem~A and the classical
Las Vergnas--Zaslavsky formula, respectively.

For real arrangements, Koizumi \cite[Theorem~6.4]{Koizumi} computes
the integral homology using face flags, which are nested sequences of
faces incident to a chamber.
Liu \cite{Liu} refines this calculation into rank-one summands indexed
by the numbers of crossings of each hyperplane and a terminal chamber,
and determines the mod-$2$ cup product.
To prove Theorem~B, we extend these results to simple oriented matroids
in Section~\ref{sec:homology}.
We replace the geometric
vanishing theorem used in Koizumi's localization argument by a proof
using tope intervals and finite posets.
The filtration then uses the oriented matroid form of the
Varchenko--Gelfand theorem developed in \cite{Moseley,SY}.

We determine the pole order of magnitude at $q=-1$ for simple orientable matroids.
For a loopless matroid $M$, let $\lambda(M)$ be the maximum length $p$
of a chain of flats
\[
 \varnothing=F_0\subsetneq F_1\subsetneq\cdots\subsetneq F_p,
 \qquad r(F_i)=i,
 \qquad |F_i\setminus F_{i-1}|\text{ odd}.
\]
The chain need not end at $E$.

\begin{theoremC}[Theorem~\ref{thm:pole-order} and
Proposition~\ref{prop:parity-degree}]
For a simple orientable matroid $M$, the order of the pole of
$\Mag(M;q)$ at $q=-1$ is $\lambda(M)$, and
\[
 \lim_{q\to-1}(1+q)^{\lambda(M)}\Mag(M;q)>0.
\]
This pole order also equals the Varchenko--Gelfand degree over $\R$
of the tope parity function $B\mapsto\prod_{e\in E}B_e$,
where the signs $B_e$ are identified with $1$ and $-1$.
\end{theoremC}

The orders of the poles of magnitude affect the asymptotic behavior of its Taylor coefficients. In Theorem~\ref{thm:eventual-alternation-counterexample}, we construct a representable simple matroid $M$ of rank $6$ such that $\Mag(M;q)$ has a pole of order $4$ at $q=-1$ and a pole of order $5$ at $q=i$.
The expansion coefficients of such a function satisfy $(-1)^\ell[q^\ell]\Mag(M;q)<0$ for infinitely many $\ell$.
This disproves Koizumi--Liu's eventual sign alternation conjecture for real hyperplane arrangements \cite[Conjecture~7.1(3)]{KL}.

The present paper is organized as follows.
Section~\ref{sec:preliminaries} fixes the conventions and recalls the
background needed for both the numerical and cohomological results.
Sections~\ref{sec:zaslavsky} and~\ref{sec:magnitude} prove Theorem~A and
develop matroid magnitude.
Section~\ref{sec:poles} proves Theorem~C and constructs a counterexample to the eventual-alternation conjecture.
Sections~\ref{sec:homology} and~\ref{sec:filtration} compute the
(co)homology and establish Theorem~B.
The numerical arguments through Section~\ref{sec:poles} are independent
of the later homology calculation.

\subsection*{Human--AI collaboration}
This work was developed through extensive interaction between the author and OpenAI's GPT-5.6 Sol and GPT-6 Astra. AI assistance included developing mathematical arguments, performing computations, and drafting and revising the manuscript. The author critically evaluated the AI-generated material, independently verified the mathematical arguments and references, and takes full responsibility for the paper.

\section{Preliminaries}\label{sec:preliminaries}

All matroids and graphs in this paper are finite.
We write $\N=\Z_{\geq0}$, and $\one_S\in\N^E$ for the indicator vector
of a subset $S\subseteq E$.
Matroids need not be realizable; an assertion about an oriented matroid
does not assume a representation by real vectors.
Our conventions for oriented matroids follow \cite{BLSWZ}.
We first recall the matroid invariants used in the numerical results,
then introduce tope graphs and the (co)homological constructions used
in the later sections.

\subsection{Matroids and characteristic polynomials}
A matroid abstracts the dependence relations of a family of vectors.
It consists of a ground set $E$ and a rank function $r_M:2^E\to\N$
such that, for all $S,T\subseteq E$,
\[
 0\leq r_M(S)\leq |S|,\qquad
 S\subseteq T\Longrightarrow r_M(S)\leq r_M(T),
\]
\[
 r_M(S\cup T)+r_M(S\cap T)\leq r_M(S)+r_M(T).
\]
The last two conditions are called monotonicity and submodularity.
For a vector configuration $(v_e)_{e\in E}$, the example to keep in mind
is $r_M(S)=\dim\operatorname{span}\{v_e:e\in S\}$.
We put $r(M)=r_M(E)$ and write $r(S)$ when $M$ is understood.
A subset is \emph{independent} if its rank equals its cardinality.
A \emph{basis} is a maximal independent subset; every basis has
cardinality $r(M)$, and $\mathcal B(M)$ denotes the set of bases.
A \emph{circuit} is a minimal dependent subset.
An element belonging to every basis is a \emph{coloop}.
The \emph{uniform matroid} $U_{r,n}$ on an $n$-element set, where
$0\leq r\leq n$, has rank function $r_{U_{r,n}}(S)=\min(r,|S|)$.

The closure of $S$ is
$\cl_M(S)=\{e\in E:r_M(S\cup\{e\})=r_M(S)\}$.
Thus closure adds the elements already spanned by $S$ in the vector
example.
A \emph{flat} is a subset $F$ with $\cl_M(F)=F$; equivalently, adding
any element outside $F$ increases its rank.
We write $\cF(M)$ for the lattice of flats, ordered by inclusion.
A \emph{loop} is an element $e$ with $r_M(\{e\})=0$.
Two distinct nonloops are parallel if together they have rank one;
a matroid is \emph{simple} if it has neither loops nor parallel elements.
Its \emph{simplification} is obtained by deleting all loops and retaining
one element from each parallel class.

For $F\subseteq E$, the \emph{restriction} $M|F$ retains only the
elements of $F$, while the \emph{contraction} $M/F$ measures rank after
the elements of $F$ have already been spanned.
Their ground sets are $F$ and $E\setminus F$, respectively, and
\[
 \begin{aligned}
 r_{M|F}(S)&=r_M(S) &&(S\subseteq F),\\
 r_{M/F}(S)&=r_M(S\cup F)-r_M(F) &&(S\subseteq E\setminus F).
 \end{aligned}
\]
For a vector configuration, contraction corresponds to taking the
images of the remaining vectors modulo the span of the vectors in $F$.
Matroids obtained by restriction and contraction are called \emph{minors}.
Contraction by a flat of a loopless matroid is loopless but need not be
simple.
For matroids $M_1,M_2$ on disjoint ground sets $E_1,E_2$, their
\emph{direct sum} is the matroid on $E_1\sqcup E_2$ with rank
$r_{M_1\oplus M_2}(S)=r_{M_1}(S\cap E_1)+r_{M_2}(S\cap E_2)$.
A matroid is \emph{connected} if it has no direct sum decomposition
with both ground sets nonempty.

Our characteristic polynomial convention is
\begin{equation}\label{eq:characteristic}
 \chi_M(x)=\sum_{S\subseteq E}(-1)^{|S|}x^{r(M)-r_M(S)}.
\end{equation}
It is zero if $M$ has a loop.
For a loopless matroid,
\[
 \chi_M(x)=\sum_{F\in\cF(M)}\mu(\varnothing,F)x^{r(M)-r(F)},
\]
where $\mu$ is the M\"obius function of the flat lattice.
It is determined recursively by $\mu(F,F)=1$ and
$\sum_{F\subseteq H\subseteq G}\mu(F,H)=0$ for flats $F\subsetneq G$,
where the sum runs over flats $H$.
The two expressions for $\chi_M$ agree by grouping subsets according
to their closure.
We write
\begin{equation}\label{eq:region-number}
 R(M)=(-1)^{r(M)}\chi_M(-1).
\end{equation}
For loopless $M$, this is a positive integer.
The empty matroid has rank zero, characteristic polynomial $1$, and
$R(\varnothing)=1$.

\subsection{Weights, initial matroids, and the motivic zeta function}
For $w\in\N^E$, let
\[
 |w|=\sum_{e\in E}w_e,\qquad
 \wt_M(w)=\max_{B\in\mathcal B(M)}\sum_{e\in B}w_e.
\]
The maximizing bases are the bases of a matroid $M_w$ on $E$, called
the \emph{initial matroid}.
In particular $r(M_w)=r(M)$.
Throughout, initial matroids use \emph{maximal} weights.
An element is a loop of $M_w$ precisely when it belongs to no
maximizing basis, so $M_w$ can have loops even if $M$ does not.

For nonnegative integer weights the \emph{upper level sets}
\begin{equation}\label{eq:upper-level-sets}
 U_j(w)=\{e\in E:w_e\geq j\},\qquad j\geq1,
\end{equation}
form a nested sequence $U_1(w)\supseteq U_2(w)\supseteq\cdots$,
eventually empty, and $w=\sum_{j\geq1}\one_{U_j(w)}$.
Each element $e$ is counted in exactly $w_e$ of these sets.
The greedy algorithm gives
\begin{equation}\label{eq:kappa}
 \kappa(w):=\wt_M(w)=\sum_{j\geq1}r_M(U_j(w)).
\end{equation}
Indeed, the weight of a basis $B$ is
$\sum_j|B\cap U_j(w)|\leq\sum_j r_M(U_j(w))$.
If $w=0$, every basis attains equality.
Otherwise, starting with a basis of the smallest nonempty upper level set and
successively extending it to bases of the larger sets, and finally
of $M$, attains every bound simultaneously.
Consequently, $B$ is maximizing if and only if
$|B\cap U_j(w)|=r_M(U_j(w))$ for every $j$.

Suppose $M$ is loopless.
Then $M_w$ is loopless if and only if every nonempty $U_j(w)$ is a flat.
To see the forward implication, if $e\in\cl_M(U_j)\setminus U_j$,
every maximizing basis contains a basis of $U_j$ and hence omits $e$:
including $e$ would make it dependent.
Thus $e$ would be a loop of $M_w$.
For the converse, suppose all nonempty upper level sets are flats.
List the distinct nonempty upper level sets in increasing order,
append $E$ if necessary, and prepend $\varnothing$:
\[
 \varnothing=F_0\subsetneq F_1\subsetneq\cdots\subsetneq F_s=E.
\]
The greedy description of the bases gives
\begin{equation}\label{eq:initial-minors}
 M_w=\bigoplus_{i=1}^s (M|F_i)/F_{i-1}.
\end{equation}
Here the $i$th summand has ground set $F_i\setminus F_{i-1}$:
a maximizing basis is built by extending a basis of $F_{i-1}$ to one
of $F_i$ at each step.
Every summand is loopless because $F_{i-1}$ is a flat in $M|F_i$.
For any matroid $M$, we set
\begin{equation}\label{eq:admissible-weights}
 \mathcal W(M)=\{w\in\N^E:M_w\text{ is loopless}\}.
\end{equation}
The empty chain convention covers $E=\varnothing$.
For nonempty loopless $M$, the weight $0$ corresponds to the chain
$\varnothing\subsetneq E$ and $M_0=M$.
For example, if $\varnothing\subsetneq F\subsetneq E$ is a flat and
$w=2\one_F+\one_E$, then $U_1(w)=E$ and $U_2(w)=U_3(w)=F$.
In this case $M_w=M|F\oplus M/F$ and
$\kappa(w)=r(M)+2r(F)$.

\begin{example}\label{ex:initial-u23}
Let $M=U_{2,3}$ on $E=\{1,2,3\}$.
For $w=(2,0,0)$, the maximizing bases are $\{1,2\}$ and $\{1,3\}$,
so $M_w=U_{1,1}\oplus U_{1,2}$, with summands on $\{1\}$ and
$\{2,3\}$.
This initial matroid is loopless but has the parallel pair $2,3$;
the nonempty upper level sets are $U_1(w)=U_2(w)=\{1\}$, a flat.
For $v=(2,1,0)$, only $\{1,2\}$ maximizes the basis weight, so $3$ is
a loop of $M_v$.
Here $U_1(v)=\{1,2\}$ is not a flat, since its closure is $E$.
Thus $w\in\mathcal W(M)$, whereas $v\notin\mathcal W(M)$.
\end{example}

Following Jensen--Kutler--Usatine \cite[Definition~1.1]{JKU}, define
the \emph{motivic zeta function} of a matroid $M$ by
\begin{equation}\label{eq:zeta-definition}
 Z_M(x,T)=\sum_{w\in\N^E}
 \chi_{M_w}(x)x^{-r(M)-\wt_M(w)}T^{|w|}
 \in\Z[x^{\pm1}][[T]].
\end{equation}
The letter $x$ replaces the variable called $q$ in \cite{JKU}; we reserve
$q$ for magnitude.
This is a formal power series in $T$ with Laurent polynomial
coefficients in $x$.
There are only finitely many weights of any fixed total weight, so
each coefficient is a finite sum; in particular, substituting $x=-1$
is well-defined coefficientwise.
Weights outside $\mathcal W(M)$ contribute zero, and
\begin{equation}\label{eq:zeta-constant}
 Z_M(x,0)=x^{-r(M)}\chi_M(x),
\end{equation}
because $w=0$ is the only vector with $|w|=0$.
A \emph{flag of flats} is a chain of flats under inclusion.
For a weight contributing to $Z_M$, its upper level sets give such a
flag, possibly with repetitions, and their sum of indicator vectors
recovers $w$.
The distinct flats determine $M_w$ through
\eqref{eq:initial-minors}, while their repetitions determine the
exponents $|w|$ and $\wt_M(w)$.

\subsection{Covectors and tope graphs}
For a real central hyperplane arrangement $\{H_e\}_{e\in E}$,
choose a linear form $\alpha_e$ with $H_e=\ker\alpha_e$.
A point $v$ determines the sign vector
$(\operatorname{sign}\alpha_e(v))_{e\in E}$.
Each nonempty set of points with a fixed sign vector is a face of the
arrangement:
the zero entries record the hyperplanes containing the face, and the
nonzero entries record which side of the other hyperplanes it lies on.
Covectors abstract these sign vectors and their incidence relations.

For sign vectors $X,Y\in\{+,0,-\}^E$, their composition and separation
set are
\[
 (X\circ Y)_e=\begin{cases}X_e&X_e\ne0,\\Y_e&X_e=0,\end{cases}
 \qquad S(X,Y)=\{e:X_e=-Y_e\ne0\}.
\]
Thus composition fills the zero entries of $X$ with entries of $Y$,
whereas $S(X,Y)$ records coordinates with opposite nonzero signs.
An \emph{oriented matroid} $\OM=(E,\cL)$ consists of a set of
\emph{covectors} $\cL\subseteq\{+,0,-\}^E$ satisfying:
\begin{enumerate}
\item $0\in\cL$;
\item if $X\in\cL$, then $-X\in\cL$;
\item if $X,Y\in\cL$, then $X\circ Y\in\cL$;
\item if $X,Y\in\cL$ and $e\in S(X,Y)$, there is $Z\in\cL$ with
      $Z_e=0$ and $Z_f=(X\circ Y)_f$ for every $f\notin S(X,Y)$.
\end{enumerate}
The last axiom is called \emph{strong elimination}.
For an arrangement, it describes passing through a separating
hyperplane while retaining the signs outside the separation set.
The zero sets $Z(X)=\{e:X_e=0\}$ are precisely the flats of its
underlying matroid $M$.
In the arrangement example, $M$ is the matroid of the linear forms
$(\alpha_e)_{e\in E}$.
We also write $r(\OM)=r(M)$.
We call $M$ \emph{orientable} if it underlies an oriented matroid.
The covectors of the restriction $\OM|F$ are $X|_F$ for $X\in\cL$.
Those of the contraction $\OM/F$ are $X|_{E\setminus F}$ for
$X\in\cL$ vanishing on $F$.
Their underlying matroids are $M|F$ and $M/F$, respectively.

The covector order is $X\leq Y$ if $X_e\in\{0,Y_e\}$ for every $e$.
In the arrangement example, this means that the face indexed by $X$
lies in the closure of the face indexed by $Y$.
Its maximal elements are the \emph{topes}, and their set is $\cT(\OM)$.
If $M$ is loopless, the topes are exactly the covectors with no zero
coordinates; for an arrangement, these are its \emph{chambers}, the
connected components of the complement of the hyperplanes.
The \emph{tope graph} $G(\OM)$ joins two topes when they cover a common
covector.
Here $B$ covers $X$ if $X<B$ and no covector lies strictly between
them; geometrically, the two chambers share a face of codimension one.
When $M$ is simple, adjacency means differing in one coordinate, and
\begin{equation}\label{eq:tope-distance}
 d(B,C)=|S(B,C)|.
\end{equation}
Here $d$ is the shortest-path distance in the graph.
Equivalently, there is a path from $B$ to $C$ that changes each
coordinate in $S(B,C)$ exactly once and changes no other coordinate.
Thus the inclusion of the tope graph in the cube $\{+,-\}^E$ preserves
distances; such a subgraph is called \emph{isometric}.
Simplicity is needed here: parallel coordinates change together and
are counted only once by the unweighted graph distance.

For a covector $X$, its localization is the restriction
$\OM_X=\OM|Z(X)$.
The topes above $X$ have their signs fixed by $X$ outside $Z(X)$;
restriction to $Z(X)$ identifies them with $\cT(\OM_X)$.
Conversely, given a local tope $U$, choose $Y\in\cL$ with
$Y|_{Z(X)}=U$; then $X\circ Y$ is the corresponding tope above $X$,
independently of the choice of $Y$.
For an arrangement, this keeps just the hyperplanes containing the
face $X$ and describes the chambers incident to that face.
This accounts for the geometric term ``localization'' despite its
being a restriction of the underlying matroid.

For example, take the three lines defined by the forms
$\alpha_1(x,y)=y$, $\alpha_2(x,y)=x$, and $\alpha_3(x,y)=x-y$.
The ray $\{(x,0):x>0\}$ is the face $X=(0,+,+)$, with $Z(X)=\{1\}$.
The two topes of $\OM_X=\OM|\{1\}$ record the two chambers above this
fixed ray, whose signs are $(+,+,+)$ and $(-,+,+)$.
In contrast, $\OM/\{1\}$ describes the arrangement restricted to the
line $y=0$; its two topes $(+,+)$ and $(-,-)$ record the two rays on
that line, namely the faces $X$ and $-X$ with zero set $\{1\}$.
Thus localization counts chambers incident to a fixed face, while
contraction counts faces with a fixed zero set.

For loopless $M$, the Las Vergnas--Zaslavsky formula is
\begin{equation}\label{eq:classical-zaslavsky}
 |\cT(\OM)|=R(M).
\end{equation}
We use the oriented matroid form as in \cite[\S4.6]{BLSWZ}; the
arrangement theorem originates in \cite{Zaslavsky}.

\subsection{Magnitude and its homological refinement}
For a finite connected graph $G$, with shortest-path distance $d$,
let $q$ be an indeterminate and put
\[
 K_G(q)=(q^{d(B,C)})_{B,C\in V(G)}.
\]
This is the \emph{similarity matrix} of $G$.
Since $K_G(0)=I$, its determinant has constant term $1$, so the matrix
is invertible over $\Q(q)$.
Its \emph{weighting} is the column vector
$w_G(q)=K_G(q)^{-1}\one$, where $\one$ is the all-ones column vector.
Equivalently, it is the unique solution of
\[
 \sum_{C\in V(G)}q^{d(B,C)}w_G(q)_C=1
 \qquad(B\in V(G)).
\]
Here a weighting is distinct from a weight vector used to define an
initial matroid.
Leinster's magnitude \cite{Leinster} is
\begin{equation}\label{eq:graph-magnitude}
 \Mag(G;q)=\one^{\mathsf T}K_G(q)^{-1}\one
          =\sum_{B\in V(G)}w_G(q)_B.
\end{equation}
It is a rational function regular at $0$, with constant term $|V(G)|$.

We recall the magnitude homology of Hepworth--Willerton \cite{HW} and the cup
product of Hepworth \cite{Hepworth}.
The group $\MC_{k,\ell}(G;\Z)$ is freely generated by the tuples
$(B_0,\ldots,B_k)$ of vertices of $G$ with consecutive terms distinct and
$\sum_i d(B_{i-1},B_i)=\ell$.
The vertices in a tuple need not be adjacent, and nonconsecutive
vertices may coincide.
The index $k$ is the homological degree, while $\ell$ records the total
length of the tuple.
The differential $\partial\colon \MC_{k,\ell}(G;\Z)\to \MC_{k-1,\ell}(G;\Z)$ is given by
\begin{equation}\label{eq:magnitude-differential}
 \partial(B_0,\ldots,B_k)=
 \sum_{i=1}^{k-1}(-1)^i\partial_i(B_0,\ldots,B_k),
\end{equation}
where $\partial_i$ deletes $B_i$ when
$d(B_{i-1},B_{i+1})=d(B_{i-1},B_i)+d(B_i,B_{i+1})$ and is zero otherwise.
Thus an interior vertex can be deleted precisely when it lies on a
shortest path between its neighbors, so that deletion preserves length.
For example, if $A,B,C$ occur in this order on a shortest path, then
$\partial(A,B,C)=-(A,C)$; by contrast, $\partial(A,B,A)=0$.
Deleting an endpoint always decreases length, so there are no endpoint
terms.
Write $\MH_{k,\ell}$ for its homology and $\MH^{k,\ell}(G;\Bbbk)$
for the cohomology of the dual complex with coefficients in a field
$\Bbbk$.
Explicitly, a cochain in bidegree $(k,\ell)$ is a $\Bbbk$-valued
function on the generating tuples of $\MC_{k,\ell}$, and its
coboundary is given by precomposition with $\partial$.
The cup product is induced by
\begin{equation}\label{eq:cup-definition}
 (f\smile g)(B_0,\ldots,B_{a+b})
 =f(B_0,\ldots,B_a)g(B_a,\ldots,B_{a+b})
\end{equation}
for cochains of degrees $a,b$, extending each cochain by zero outside
its length component.
The formula splits a tuple at the shared vertex $B_a$.
If the cochains have lengths $\ell,m$, their product has bidegree
$(a+b,\ell+m)$; this product is generally noncommutative.
At length zero only single-vertex tuples occur, so
$\MH^{0,0}(G;\Bbbk)=\Fun(V(G),\Bbbk)$ with pointwise multiplication.
Hepworth--Willerton proved the following Euler characteristic identity:
\begin{equation}\label{eq:magnitude-euler}
 \Mag(G;q)=\sum_{k,\ell\geq0}(-1)^k
 \dim_{\Bbbk}\MH^{k,\ell}(G;\Bbbk)q^\ell.
\end{equation}
This is an equality of formal power series at $q=0$.
The magnitude homology vanishes for $k>\ell$, so
the coefficient of each $q^\ell$ is a finite sum.
The identity expresses magnitude as the Euler characteristic in each
length, which is the sense in which magnitude (co)homology refines it.

Now let $G=G(\OM)$ be the tope graph of a simple oriented matroid.
Equation~\eqref{eq:tope-distance} refines length to the
\emph{crossing vector}
\begin{equation}\label{eq:crossing-vector}
 w(B_0,\ldots,B_k)=\sum_{i=1}^k\one_{S(B_{i-1},B_i)}\in\N^E.
\end{equation}
The coordinate $w_e$ counts how many successive pairs change sign at
$e$, and $|w|=\ell$.
The differential $\partial$ preserves the crossing vector $w$ and both endpoints.
For fixed $w\in\N^E$ and topes $A,B$, let
$\MC_{k,w}(G)_{A\to B}$ be the subgroup of
$\MC_{k,|w|}(G;\Z)$ freely generated by the tuples $(B_0,\ldots,B_k)$
satisfying
\[
 B_0=A,\qquad B_k=B,\qquad w(B_0,\ldots,B_k)=w.
\]
The differential maps this
subgroup into $\MC_{k-1,w}(G)_{A\to B}$.
These groups therefore form a subcomplex, denoted
$\MC_{*,w}(G)_{A\to B}$, where $*$ ranges over homological degrees
while $w,A,B$ remain fixed.
A bullet in place of an endpoint means taking the direct sum over
that endpoint; for instance,
$\MC_{*,w}(G)_{\bullet\to B}=\bigoplus_A\MC_{*,w}(G)_{A\to B}$.
Omitting both endpoints means taking the direct sum over both $A$ and $B$.
The notations $\MH_{k,w}$ and $\MH^{k,w}$ refer to this finer grading.
If $\eps_w$ reverses the coordinates on which $w_e$ is odd, then the
initial tope of a chain of crossing vector $w$ and terminal tope $B$
must be $\eps_w B$, since a coordinate has different initial and
terminal signs exactly when it changes an odd number of times.
For a subset $F$, write $\eps_F B$ for sign reversal on $F$, and
$B\setminus F$ for the sign vector obtained by setting those coordinates
to zero.
These operations are defined on sign vectors; they need not produce
a tope or a covector of the given oriented matroid.

The notation linking matroid weights to chains is collected in
Table~\ref{tab:crossing-notation}.
The support of a weight is $\supp(w)=\{e\in E:w_e>0\}$.
\begin{table}[htbp]
\caption{Weights and crossing vectors.}
\label{tab:crossing-notation}
\centering
\begin{tabular}{@{}lp{0.65\textwidth}@{}}
\toprule
Notation & Meaning \\
\midrule
$w\in\N^E$ & A matroid weight; for a chain, its coordinate crossing counts. \\
$|w|=\sum_e w_e$ & The length $\ell$ of a chain with crossing vector $w$. \\
$\kappa(w)=\wt_M(w)$ & The maximum basis weight; the homological degree
of the nonzero components in Theorem~\ref{thm:homology}. \\
$\eps_w$ & Sign reversal in the coordinates where $w_e$ is odd;
a chain ending at $B$ starts at $\eps_w B$. \\
$\mathcal W(M)$ & Weights whose initial matroid $M_w$ is loopless. \\
\bottomrule
\end{tabular}
\end{table}

\subsection{The Varchenko--Gelfand filtration}\label{subsec:vg-background}
Let $\mathcal N$ be a loopless oriented matroid with underlying matroid
$N$ of rank $d$, allowing parallel elements.
This generality is needed because the initial matroids above can have
parallel elements even when the original matroid is simple.
Consider the algebra
\[
 V(\mathcal N)=\Fun(\cT(\mathcal N),\Ftwo)
\]
of functions on topes, with pointwise addition and multiplication.
For each $e$, the \emph{Heaviside function} $h_e$ is the indicator of
the positive side: $h_e(B)=1$ if $B_e=+$ and $h_e(B)=0$ if $B_e=-$.
The \emph{Varchenko--Gelfand filtration} is the increasing degree
filtration
\begin{equation}\label{eq:vg-filtration}
 P_pV(\mathcal N)=\operatorname{span}_{\Ftwo}
 \left\{\prod_{e\in S}h_e:|S|\leq p\right\},\qquad P_{-1}=0.
\end{equation}
Repeated factors are unnecessary because $h_e^2=h_e$.
Thus $P_pV(\mathcal N)$ consists of the functions that can be
represented by polynomials of degree at most $p$ in the $h_e$.
For example, $P_0$ consists of constant functions, and $P_1$ is spanned
by $1$ and the individual $h_e$.
Products satisfy $P_pV\cdot P_qV\subseteq P_{p+q}V$.
The \emph{associated graded algebra}
\[
 \gr_P V=\bigoplus_{p\geq0}\gr_p^P V,\qquad
 \gr_p^P V=P_pV/P_{p-1}V,
\]
records the degree-$p$ functions modulo those expressible in lower
degree.
This definition also makes sense over $\Z$ or $\R$.
We specify real coefficients when using it for pole orders.

The \emph{Orlik--Solomon algebra} $\OS^\bullet(N;\Ftwo)$ \cite{OS} is the
quotient of the exterior algebra on degree-one generators $a_e$ by the
ideal generated by the circuit boundaries
\[
 \partial(a_{e_1}\cdots a_{e_m})
 =\sum_{i=1}^m a_{e_1}\cdots\widehat{a_{e_i}}\cdots a_{e_m},
 \qquad \{e_1,\ldots,e_m\}\text{ a circuit}.
\]
The hat means that the indicated factor is omitted, and each displayed
boundary is set equal to zero in the quotient.
The generators commute in characteristic two but still satisfy $a_e^2=0$.
For example, a three-element circuit $\{e,f,g\}$ imposes the relation
$a_fa_g+a_ea_g+a_ea_f=0$.
The oriented matroid form of the Varchenko--Gelfand theorem gives
\begin{equation}\label{eq:vg-os}
 \gr_P V(\mathcal N)\cong\OS^\bullet(N;\Ftwo),\qquad
 [h_e]\longleftrightarrow a_e.
\end{equation}
Here is a direct argument over $\Ftwo$.
The orientation assigns to each circuit $C$ a partition
$C=C^+\sqcup C^-$, defined up to interchanging the two parts; this is
a \emph{signed circuit}.
For real vectors, the two parts record the signs of the coefficients
in a minimal linear dependence.
Circuit--covector orthogonality says that a tope can neither agree
everywhere nor disagree everywhere with these signs on $C$.
Thus both polynomials
\[
 \prod_{e\in C^+}h_e\prod_{e\in C^-}(1+h_e),\qquad
 \prod_{e\in C^+}(1+h_e)\prod_{e\in C^-}h_e
\]
vanish on the topes: the first detects agreement with all the circuit
signs, and the second detects disagreement with all of them.
In their sum the degree-$|C|$ terms cancel, leaving degree at most
$|C|-1$, with homogeneous part in that degree
$\sum_{e\in C}\prod_{f\in C\setminus\{e\}}h_f$.
Passing to the associated graded therefore gives the circuit boundary
relation for the classes $[h_e]$.
Likewise, $h_e^2=h_e\in P_1V$ implies $[h_e]^2=0$ in $P_2V/P_1V$.
Since the classes $[h_e]$ generate $\gr_P V$, the assignment
$a_e\mapsto[h_e]$ induces a surjection
$\OS^\bullet(N;\Ftwo)\twoheadrightarrow\gr_P V$.

To show that there are no further relations, compare dimensions.
Fix an ordering of the ground set; a \emph{broken circuit} is a circuit
with its least element removed.
The no-broken-circuit basis theorem gives a basis of $\OS^p$ indexed
by the $p$-element subsets containing no broken circuit.
Their number is $(-1)^p$ times the coefficient of $x^{d-p}$ in
$\chi_N(x)$; see \cite[Theorem~5.3]{SY}.
In particular, the total dimension of $\OS^\bullet(N;\Ftwo)$ is $R(N)$.
On the other hand, the indicator of a single tope $B$ is
\[
 \delta_B=\prod_{e:B_e=+}h_e\prod_{e:B_e=-}(1+h_e).
\]
These indicators form a basis of $V(\mathcal N)$, so the filtration is
exhaustive and $\gr_P V$ also has total dimension
$\dim V(\mathcal N)=|\cT(\mathcal N)|=R(N)$, by
\eqref{eq:classical-zaslavsky}.
The surjection is therefore an isomorphism.
Consequently, if
$b_p(N)=\dim_{\Ftwo}\gr_p^P V(\mathcal N)$, then
\begin{equation}\label{eq:vg-characteristic}
 \chi_N(x)=\sum_{p=0}^d(-1)^p b_p(N)x^{d-p},\qquad
 P_dV(\mathcal N)=V(\mathcal N).
\end{equation}
Thus polynomial degree on tope functions recovers the characteristic
polynomial, although the tope set itself depends on the orientation.
The same formula for the graded dimensions holds over $\R$, and the
integral graded pieces are free of the corresponding ranks; see
\cite[Theorem~5.9]{Moseley} and
\cite[Proposition~6.4 and Corollary~6.5]{SY}, respectively.
The original construction for arrangements is due to
Varchenko--Gelfand \cite{VG}.
Section~\ref{sec:poles} uses polynomial degree and the graded-dimension
formula over $\R$ to determine pole orders.
Section~\ref{sec:filtration} uses the algebra isomorphism
\eqref{eq:vg-os} over $\Ftwo$ to describe the associated graded of
magnitude cohomology.

\section{The magnitude-refined Las Vergnas--Zaslavsky theorem}\label{sec:zaslavsky}

The Las Vergnas--Zaslavsky theorem identifies the number of topes of a
loopless oriented matroid with a specialization of the characteristic
polynomial of its underlying matroid.  By taking the distances between
topes into account, we refine this theorem to a formula for magnitude.
The corresponding refinement on the matroid side is the motivic zeta
function.

\begin{theorem}[Magnitude-refined Las Vergnas--Zaslavsky theorem]\label{thm:zaslavsky}
Let $\OM$ be a simple oriented matroid with underlying matroid $M$.
Then
\begin{equation}\label{eq:zaslavsky}
  \Mag(G(\OM);q)=Z_M(-1,q).
\end{equation}
The equality holds both in $\mathbb Z[[q]]$ and as an equality of
rational functions.
\end{theorem}

At $q=0$, the similarity matrix of a graph is the identity, whereas only
the zero weight contributes to the motivic zeta function.  Thus the
constant term of \eqref{eq:zaslavsky} is the Las Vergnas--Zaslavsky
formula \cite[\S4.6]{BLSWZ}:
\[
  |\cT(\OM)|=(-1)^{r(M)}\chi_M(-1).
\]
Our proof uses the Las Vergnas--Zaslavsky theorem for contractions of
$\OM$.  A key ingredient is a face decomposition extending
Koizumi--Liu's formula \cite[Theorem~3.7]{KL}.
Grouping covectors by their zero sets gives the same recursion as the
motivic zeta function.

\subsection{The Euler characteristic of a collection of faces}

Fix a simple oriented matroid $\OM=(E,\cL)$ of rank $r$.  For a tope
$D$ and a covector $F$, put
\[
  N_D(F)=\{e\in E:F_e\ne D_e\}.
\]
In particular, every zero coordinate of $F$ belongs to $N_D(F)$.
Write $\OM_F=\OM|_{Z(F)}$ for the localization at $F$.
We will construct a weighting of the whole tope graph by taking a
signed sum of weightings on the topes incident with each face.
At a tope $D$, the weighting equation reduces to the identity in the following lemma.

\begin{lemma}\label{lem:face-euler}
For every $D\in\cT(\OM)$,
\begin{equation}\label{eq:face-euler}
  \sum_{F\in\cL}(-1)^{r(\OM_F)}q^{|N_D(F)|}=1.
\end{equation}
\end{lemma}

For a nonempty real arrangement and a subset $A\subseteq E$,
the condition $A\subseteq N_D(F)$ selects
the faces lying on the $-D$ side, or on the boundary, of each
hyperplane indexed by $A$.  Intersecting these closed half-spaces
with a sphere in the essential space gives the whole sphere when $A=\varnothing$ and a
closed ball otherwise.  This is the convexity argument of
\cite[Lemma~3.11]{KL}.  The proof below obtains the same Euler
characteristics without a realization, using two facts:
the tope order admits a shelling, and a nonempty proper initial
segment of that shelling is a ball.

\begin{proof}
Expand $q^{|N_D(F)|}$ as a polynomial in $q-1$.  The left-hand side is
\[
  \sum_{A\subseteq E}(q-1)^{|A|}e_A,
  \qquad
  e_A=\sum_{\substack{F\in\cL\\A\subseteq N_D(F)}}
           (-1)^{r(\OM_F)}.
\]
We prove that $e_\varnothing=1$ and that $e_A=0$ for every nonempty
$A$.  The assertion is immediate when $E=\varnothing$, so assume
$r>0$.

The nonzero covectors form the face poset of a regular CW decomposition
$K(\OM)$ of the $(r-1)$-sphere
\cite[Theorem~4.3.3]{BLSWZ}.
If $B$ is a fixed tope, the tope
poset is ordered by
\[
  C\preceq_B C'
  \quad\Longleftrightarrow\quad
  S(B,C)\subseteq S(B,C').
\]
Every linear extension of this poset is a shelling of $K(\OM)$
\cite[Theorem~4.3.3]{BLSWZ}.  We use this theorem with $B=-D$.
The set
\[
  I_A=\{C\in\cT(\OM):A\cap S(-D,C)=\varnothing\}
\]
is an order ideal of the tope poset.  Let $K_A$ be the union of the
closed top-dimensional cells indexed by $I_A$.

The cells of $K_A$ are precisely the nonzero covectors $F$ satisfying
$A\subseteq N_D(F)$.  Indeed, such a covector satisfies
$F_e\in\{0,-D_e\}$ for $e\in A$, so $F\circ(-D)$ belongs to $I_A$
and lies above $F$.  Conversely, every face of a tope in $I_A$ has the
stated coordinate property.  If $A=\varnothing$, we obtain the whole
sphere.  If $A\ne\varnothing$, then $-D\in I_A$ and $D\notin I_A$.
Choose a linear extension in which the ideal $I_A$ occurs first.
The union of a nonempty proper initial segment of a shelling of a
regular cellular sphere is a closed ball
\cite[Proposition~4.7.26]{BLSWZ}.  Consequently,
\[
  \chi(K_A)=
  \begin{cases}
    1+(-1)^{r-1},&A=\varnothing,\\
    1,&A\ne\varnothing.
  \end{cases}
\]
The cell corresponding to a nonzero $F$ has dimension
$r-r(\OM_F)-1$.  The zero covector contributes $(-1)^r$ to every
$e_A$.  Hence
\[
  e_A=(-1)^r+(-1)^{r-1}\chi(K_A)
      =(-1)^{r-1}\bigl(\chi(K_A)-1\bigr),
\]
which gives the required values.
\end{proof}

\subsection{Local weightings and face decomposition}

The topes incident with a covector form a gated subgraph of the tope graph.
Recall that a subgraph $H$ of a graph $G$ is
\emph{gated} if every vertex $D$ of $G$ has a vertex $g(D)$ in $H$
such that
\[
  d(D,C)=d(D,g(D))+d(g(D),C)\qquad(C\in H).
\]
Thus all distances from $D$ to $H$ factor through a single vertex.
The gate identity allows us to use the weighting equations on $H$
in the calculation on $G$.

\begin{lemma}\label{lem:face-gate}
For $F\in\cL$, let $G_F$ be the subgraph of $G(\OM)$ induced by
$\cT(\OM)_F=\{C\in\cT(\OM):F\le C\}$.  Restriction to $Z(F)$
identifies $G_F$ isometrically with $G(\OM_F)$.  Moreover, $G_F$ is
gated in $G(\OM)$, with gate
\[
  g_F(D)=F\circ D,
\]
and
\begin{equation}\label{eq:gate-discrepancy}
  |N_D(F)|=|Z(F)|+d(D,g_F(D)).
\end{equation}
\end{lemma}

\begin{proof}
The tope graph of a simple oriented matroid is an isometric subgraph
of the sign cube; equivalently,
\[
  d(C,C')=|S(C,C')|
\]
for any two topes \cite[Proposition~2]{BCK}.  The restriction
$\OM_F$ is again simple.  Every tope $C\ge F$ satisfies $C_e=F_e$
for $e\in E\setminus Z(F)$ and is therefore determined by its restriction
to $Z(F)$.  Conversely, if $U$ is a tope of $\OM_F$, choose
$X\in\cL$ with $X|_{Z(F)}=U$.  Then $F\circ X$ is a tope of
$\OM$ above $F$ and restricts to $U$.  The restriction map is
therefore a bijection.  It preserves separation sets, hence
distances and adjacency.

For $C\ge F$, coordinatewise comparison gives the disjoint union
\[
  S(D,C)=S(D,F\circ D)\mathbin{\dot\cup}S(F\circ D,C).
\]
Taking cardinalities proves the gate identity.  Finally,
\[
  N_D(F)=Z(F)\mathbin{\dot\cup}S(D,F\circ D),
\]
which proves \eqref{eq:gate-discrepancy}.
\end{proof}

\begin{theorem}[Face decomposition]\label{thm:face-decomposition}
For every simple oriented matroid $\OM=(E,\cL)$,
\begin{equation}\label{eq:face-decomposition}
  \Mag(G(\OM);q)
  =\sum_{F\in\cL}(-1)^{r(\OM_F)}q^{|Z(F)|}
      \Mag(G(\OM_F);q).
\end{equation}
\end{theorem}

The zero covector contributes
$(-1)^{r(\OM)}q^{|E|}\Mag(G(\OM);q)$ to the right-hand side.
When $E\ne\varnothing$, moving this term to the left gives a
recursion over proper restrictions.  At the other end, every tope
has empty zero set and contributes $1$, the magnitude of the empty
localization.

\begin{proof}
For each $F$, let $w_F$ be the magnitude weighting on
$G_F\cong G(\OM_F)$, extended by zero to all topes.  Define
\[
  \widetilde w(C)=\sum_{F\in\cL}
     (-1)^{r(\OM_F)}q^{|Z(F)|}w_F(C).
\]
For a tope $D$, the gate identity and the weighting equations give
\begin{align*}
  \sum_Cq^{d(D,C)}\widetilde w(C)
  &=\sum_F(-1)^{r(\OM_F)}q^{|Z(F)|}
       \sum_{C\ge F}q^{d(D,C)}w_F(C)\\
  &=\sum_F(-1)^{r(\OM_F)}q^{|Z(F)|+d(D,g_F(D))}
       \sum_{C\ge F}q^{d(g_F(D),C)}w_F(C)\\
  &=\sum_F(-1)^{r(\OM_F)}q^{|N_D(F)|}=1.
\end{align*}
The final equality is Lemma~\ref{lem:face-euler}.  Thus
$\widetilde w$ is the unique weighting of $G(\OM)$.  Summing it over
$C$ gives \eqref{eq:face-decomposition}.
\end{proof}

This extends the face decomposition for real central arrangements
in \cite[Theorem~3.7]{KL}.

\subsection{The common recursion over flats}

We next obtain a recursion for $Z_M$ directly from its definition.

\begin{lemma}\label{lem:zeta-flat-recursion}
Let $M$ be a loopless matroid of rank $r$.  Then
\begin{equation}\label{eq:zeta-flat-recursion}
  Z_M(x,T)=x^{-r}\sum_{A\in\cF(M)}
        \chi_{M/A}(x)T^{|A|}Z_{M|A}(x,T).
\end{equation}
\end{lemma}

\begin{proof}
For $w\in\mathbb Z_{\ge0}^E$, write $A=\operatorname{supp}(w)$.
This is the upper level set $U_1(w)$ from
\eqref{eq:upper-level-sets}.  The greedy description in
Section~\ref{sec:preliminaries} shows that every maximum-weight
basis meets $A$ in a basis of $M|A$, and that $M_w$ has a loop
if $A$ is not a flat.  Such weights contribute zero to $Z_M$.

For a flat $A$, remove one unit of weight from every coordinate
in the support by writing $w|_A=\boldsymbol1+v$.
A basis of maximum weight is the union of a $v$-maximum basis of
$M|A$ and a basis of $M/A$.  Conversely, every such union is a
basis of maximum $w$-weight.  Therefore
\[
  M_w=(M|A)_v\oplus M/A,
  \qquad
  \wt_M(w)=r_M(A)+\wt_{M|A}(v).
\]
The contribution of all weights with positive support $A$ is
\begin{align*}
  &\sum_{v\in\mathbb Z_{\ge0}^A}
     \chi_{(M|A)_v}(x)\chi_{M/A}(x)
     x^{-r-r_M(A)-\wt_{M|A}(v)}T^{|A|+|v|}\\
  &\hspace{35mm}=x^{-r}\chi_{M/A}(x)T^{|A|}Z_{M|A}(x,T).
\end{align*}
Summing over the flats proves the identity.  Every rearrangement is
valid coefficientwise in $T$, since a fixed value of $|w|$ admits
only finitely many nonnegative integral weights.
\end{proof}

\begin{proof}[Proof of Theorem~\ref{thm:zaslavsky}]
Covector zero sets are precisely the flats of the underlying
matroid.  For a flat $A$, restriction to $E\setminus A$ gives a
bijection
\[
  \{F\in\cL:Z(F)=A\}\longrightarrow\cT(\OM/A).
\]
Indeed, contraction retains exactly those covectors which vanish
on $A$, and $M/A$ is loopless because $A$ is a flat.  Grouping
\eqref{eq:face-decomposition} by zero sets gives
\[
  \Mag(G(\OM);q)=\sum_{A\in\cF(M)}
     (-1)^{r_M(A)}|\cT(\OM/A)|q^{|A|}
     \Mag(G(\OM|A);q).
\]
The Las Vergnas--Zaslavsky theorem for the loopless oriented matroid
$\OM/A$ yields
\[
  (-1)^{r_M(A)}|\cT(\OM/A)|
   =(-1)^r\chi_{M/A}(-1).
\]
Consequently, the preceding identity and
\eqref{eq:zeta-flat-recursion} at $(x,T)=(-1,q)$ give the same
recursion.

For completeness, this recursion determines the invariant
uniquely.  Both invariants equal $1$ on the empty matroid.  If
$E\ne\varnothing$, moving the term $A=E$ to the left gives
\begin{equation}\label{eq:common-flat-recursion}
  \bigl(1-(-1)^rq^{|E|}\bigr)I(\OM;q)
   =\sum_{\substack{A\in\cF(M)\\A\subsetneq E}}
       (-1)^r\chi_{M/A}(-1)q^{|A|}I(\OM|A;q).
\end{equation}
Every restriction on the right is simple and has a smaller ground
set.  The coefficient on the left is a unit in $\mathbb Z[[q]]$.
Induction proves equality of the formal power series.  The same
recursion proves that $Z_M(-1,q)$ is rational and regular at zero,
so the equality also holds in $\mathbb Q(q)$.
\end{proof}

\begin{corollary}\label{cor:orientation-independence}
The magnitude of the tope graph of a simple oriented matroid
depends only on its underlying matroid.
\end{corollary}

\section{Magnitude of matroids}\label{sec:magnitude}

Theorem~\ref{thm:zaslavsky} expresses a metric invariant of topes
using only unoriented matroid data.  Motivated by this formula, we
define magnitude for an arbitrary matroid as follows.

\begin{definition}\label{def:matroid-magnitude}
The \emph{magnitude} of a matroid $M$ is
\[
  \Mag(M;q):=Z_M(-1,q)\in\mathbb Z[[q]].
\]
\end{definition}

If $M$ has a loop, then every initial matroid $M_w$ has that loop
and $\Mag(M;q)=0$.  The empty matroid has magnitude $1$.
Recall that for a loopless matroid $N$,
\[
  R(N)=(-1)^{r(N)}\chi_N(-1).
\]
The alternating signs of the characteristic polynomial imply
$R(N)>0$, with $R(\varnothing)=1$.  For an orientable matroid it
counts the topes of any orientation.  In this section $R(N)$ is
used for arbitrary loopless matroids, whether or not they admit
an orientation.

\subsection{Algebraic properties}

\begin{theorem}\label{thm:magnitude-properties}
Let $M$ be a loopless matroid of rank $r$ on an $n$-element ground
set $E$.
\begin{enumerate}
\item\label{item:mag-rational} The series $\Mag(M;q)$ is a rational
function whose denominator divides
\begin{equation}\label{eq:flat-denominator}
  \prod_{\substack{A\in\cF(M)\\A\ne\varnothing}}
     \bigl(1-(-1)^{r_M(A)}q^{|A|}\bigr).
\end{equation}
\item\label{item:mag-coefficients} The first two coefficients of
$\Mag(M;q)$ are
\[
  \Mag(M;0)=R(M),\qquad
  [q]\Mag(M;q)=-2
      \sum_{\substack{e\in E\\\{e\}\in\cF(M)}}R(M/e).
\]
\item\label{item:mag-reciprocity} As rational functions,
\begin{equation}\label{eq:magnitude-reciprocity}
  \Mag(M;q^{-1})=q^n\Mag(M;q).
\end{equation}
\item\label{item:mag-one} The magnitude $\Mag(M;q)$ is regular at $q=1$, with
\begin{equation}\label{eq:magnitude-one}
  \Mag(M;1)=1.
\end{equation}
\item\label{item:mag-palindromic} Write $\Mag(M;q)=P_M(q)/Q_M(q)$
in lowest terms, with $P_M,Q_M\in\mathbb Z[q]$ and $Q_M(0)=1$.
Then both polynomials are palindromic,
\[
  \deg Q_M-\deg P_M=n,
\]
and all roots of $Q_M$ are roots of unity other than $1$.
\end{enumerate}
\end{theorem}

Parts~\ref{item:mag-reciprocity}--\ref{item:mag-palindromic}
extend the corresponding results for real central arrangements
in \cite[Propositions~3.1--3.3]{KL}.

\begin{proof}
\emph{Rationality and low-degree coefficients.}
Specializing Lemma~\ref{lem:zeta-flat-recursion} and isolating the
term $A=E$ gives the recursion
\begin{equation}\label{eq:matroid-flat-recursion}
  \bigl(1-(-1)^rq^n\bigr)\Mag(M;q)
    =\sum_{\substack{A\in\cF(M)\\A\subsetneq E}}
       (-1)^{r_M(A)}R(M/A)q^{|A|}\Mag(M|A;q).
\end{equation}
For every flat $A$ of $M$, the flats of the restriction $M|A$ are
precisely the flats of $M$ contained in $A$.
Induction on $n$ therefore proves both
rationality and the denominator bound in
\eqref{eq:flat-denominator}.  Taking constant terms gives $R(M)$.
For the linear coefficient, the only weights of total size one
are $\boldsymbol1_{\{e\}}$.  A nonflat singleton contributes zero
by the proof of Lemma~\ref{lem:zeta-flat-recursion}.  For a flat
singleton,
\[
  M_{\boldsymbol1_{\{e\}}}=U_{1,1}\oplus M/e,
\]
whose contribution is $-2R(M/e)$.  This proves
parts~\ref{item:mag-rational} and~\ref{item:mag-coefficients}.

\emph{Symmetry.}
Assume $E\ne\varnothing$.  We use two results of
Jensen--Kutler--Usatine about their reduced motivic zeta function
$\overline Z_M(x,T)$: the comparison formula and the functional
equation \cite[Proposition~3.3 and Theorem~1.6]{JKU},
\begin{align*}
  Z_M(x,T)&=\frac{x^{-1}(x-1)}{1-x^{-r}T^n}
               \overline Z_M(x,T),\\
  \overline Z_M(x^{-1},T^{-1})&=x^{r-1}\overline Z_M(x,T).
\end{align*}
Only these two identities are needed here.  The substitutions
involving inverse variables are operations on rational functions,
not substitutions in formal power series.  Simplifying gives
$Z_M(x^{-1},T^{-1})=T^nZ_M(x,T)$.  At $x=-1$ this is
\eqref{eq:magnitude-reciprocity}.  The empty matroid satisfies the
same identity directly.

\emph{Regularity at $q=1$.}
For even $r$, the factor $1-(-1)^rq^n$ in
\eqref{eq:matroid-flat-recursion} vanishes at $1$.  To show that
this apparent pole cancels and to evaluate the quotient, we will
compute the value and first derivative of the numerator.  Set
\[
  a_A=(-1)^r\chi_{M/A}(-1)
      =(-1)^{r_M(A)}R(M/A).
\]
We first prove two identities in the lattice of flats:
\begin{equation}\label{eq:mobius-moments}
  \sum_{A\in\cF(M)}a_A=1,
  \qquad
  \sum_{A\in\cF(M)}|A|a_A=-n.
\end{equation}
Let $\mu$ be the M\"obius function of $\cF(M)$.  Expanding the
characteristic polynomials and interchanging the sums gives
\begin{align*}
  \sum_{A\in\cF(M)}a_A
  &=\sum_{B\in\cF(M)}(-1)^{r_M(B)}
      \sum_{\substack{A\in\cF(M)\\A\subseteq B}}\mu(A,B)\\
  &=\sum_{B\in\cF(M)}(-1)^{r_M(B)}
      \boldsymbol1_{B=\varnothing}=1.
\end{align*}
For the second identity, write $|A|=\sum_{e\in E}\boldsymbol1_{e\in A}$.
Since a flat contains $e$ if and only if it contains
$\operatorname{cl}_M(e)$, the same expansion yields
\begin{align*}
  \sum_{A\in\cF(M)}|A|a_A
  &=\sum_{e\in E}
      \sum_{\substack{B\in\cF(M)\\B\supseteq\operatorname{cl}_M(e)}}
        (-1)^{r_M(B)}
        \sum_{\substack{A\in\cF(M)\\
              \operatorname{cl}_M(e)\subseteq A\subseteq B}}\mu(A,B)\\
  &=\sum_{e\in E}(-1)^{r_M(\operatorname{cl}_M(e))}=-n.
\end{align*}
Here we used the M\"obius identity
\[
  \sum_{H\subseteq A\subseteq B}\mu(A,B)=\boldsymbol1_{H=B}
\]
for flats $H\subseteq B$, and the fact that every singleton closure
has rank one because $M$ is loopless.

Proceed by induction on $n$; the case $n=0$ is immediate.
For a proper flat $A$, the induction hypothesis and differentiation
of \eqref{eq:magnitude-reciprocity} give
\begin{equation}\label{eq:restriction-one-derivative}
  \Mag(M|A;1)=1,
  \qquad
  \left.\frac{d}{dq}\Mag(M|A;q)\right|_{q=1}=-\frac{|A|}{2}.
\end{equation}
The right-hand side of \eqref{eq:matroid-flat-recursion}, denoted
$N(q)$, is consequently regular at $1$.  If $r$ is odd, then
$N(1)=1-a_E=2$, and the denominator is also $2$ at $1$.
If $r$ is even, then $N(1)=0$ and
\[
  N'(1)=\frac12\sum_{A\subsetneq E}|A|a_A=-n,
\]
where the sum is over flats and the last equality uses
\eqref{eq:mobius-moments} and $a_E=1$.  The denominator $1-q^n$
has derivative $-n$ at $1$.  Thus its zero cancels and the value
of the quotient is $1$.  This proves~\ref{item:mag-one}.

\emph{The reduced numerator and denominator.}
The denominator in \eqref{eq:flat-denominator} has integer
coefficients and constant term one.  Its irreducible factors are
cyclotomic polynomials.  By Gauss's lemma, the reduced denominator
can also be normalized in $\mathbb Z[q]$ with constant term one.
Regularity at $1$ excludes $\Phi_1$, and every other cyclotomic
polynomial is palindromic.  Hence $Q_M$ is palindromic.  Since
$\Mag(M;0)=R(M)\ne0$, symmetry gives
\[
  \Mag(M;q)\sim R(M)q^{-n}\qquad(q\longrightarrow\infty),
\]
which proves the degree difference.  Substituting the
palindromicity of $Q_M$ into
\eqref{eq:magnitude-reciprocity} then gives
\[
  q^{\deg P_M}P_M(q^{-1})=P_M(q).
\]
This proves~\ref{item:mag-palindromic}.
\end{proof}

\begin{proposition}\label{prop:magnitude-direct-sum-valuative}
Magnitude is multiplicative under direct sums and is a valuative
invariant of matroids.
\end{proposition}

\begin{proof}
If $M=M_1\oplus M_2$, write $w=(w^{(1)},w^{(2)})$.  Then
\[
  M_w=(M_1)_{w^{(1)}}\oplus(M_2)_{w^{(2)}},
  \qquad
  \wt_M(w)=\wt_{M_1}(w^{(1)})+\wt_{M_2}(w^{(2)}).
\]
Multiplicativity of the characteristic polynomial makes the
defining sum for $Z_M$ a product.  Specializing gives
\[
  \Mag(M_1\oplus M_2;q)=\Mag(M_1;q)\Mag(M_2;q).
\]

We use ``valuative'' in the sense of preserving integral relations
between indicator functions of matroid base polytopes.  Thus, if
$P(N)=\operatorname{conv}\{\boldsymbol1_B:B\text{ a basis of }N\}$,
an invariant $f$ is valuative when
\[
  \sum_i c_i\boldsymbol1_{P(M_i)}=0
  \quad\Longrightarrow\quad
  \sum_i c_if(M_i)=0
\]
for matroids on a fixed ground set and integers $c_i$.  In
particular, it respects the inclusion--exclusion relations of
matroid polytope subdivisions.  The motivic zeta function is
valuative by Ardila--Sanchez \cite[Theorem~8.11]{AS}.
The specialization $x=-1$, applied coefficientwise, is additive,
so it preserves every such relation.
\end{proof}

\begin{example}\label{ex:elementary-magnitudes}
The only flats of $U_{1,n}$ are $\varnothing$ and $E$.  Therefore
\[
  \Mag(U_{1,n};q)=\frac{2}{1+q^n}.
\]
Magnitude is consequently sensitive to parallel multiplicities.
For the Boolean matroid $B_n=U_{1,1}^{\oplus n}$, multiplicativity
gives
\[
  \Mag(B_n;q)=\left(\frac{2}{1+q}\right)^n.
\]
For $U_{2,3}$, the flats are $\varnothing$, the three singletons,
and $E$, whence
\[
  \Mag(U_{2,3};q)=\frac{6}{(1+q)(1+q+q^2)}.
\]
The latter is also the magnitude of the six-cycle, the tope graph
of three distinct lines through the origin.
\end{example}

\subsection{Positivity and a sharp bound for orientable matroids}

For a simple orientable matroid, Theorem~\ref{thm:zaslavsky}
expresses magnitude in terms of the similarity matrix of a tope graph.
We use this description to prove positivity and a sharp lower bound
depending only on the number of elements.

\begin{lemma}\label{lem:variational-magnitude}
Let $K$ be a real positive definite matrix indexed by a nonempty
finite set $V$.  For $b\in\mathbb R^V$,
\begin{equation}\label{eq:variational-magnitude}
  b^{\mathsf T}K^{-1}b
    =\max_{a\ne0}\frac{(b^{\mathsf T}a)^2}{a^{\mathsf T}Ka}.
\end{equation}
If $\varnothing\ne W\subseteq V$, then
\[
  b^{\mathsf T}K^{-1}b\ge b_W^{\mathsf T}K_W^{-1}b_W,
\]
where $K_W$ is the principal submatrix and $b_W=b|_W$.
\end{lemma}

\begin{proof}
Cauchy--Schwarz gives
\[
 (b^{\mathsf T}a)^2
 =\bigl((K^{-1/2}b)^{\mathsf T}K^{1/2}a\bigr)^2
 \le(b^{\mathsf T}K^{-1}b)(a^{\mathsf T}Ka).
\]
If $b\ne0$, equality is attained at $a=K^{-1}b$; if $b=0$,
both sides of \eqref{eq:variational-magnitude} vanish.  Restricting
the maximum to vectors supported on $W$ proves the inequality.
\end{proof}

With $b=\boldsymbol1$, Lemma~\ref{lem:variational-magnitude} says
that an isometric subspace gives a lower bound for magnitude whenever
the similarity matrix is positive definite.  For a simple oriented
matroid on $n\geq2$ elements, a shortest path from a tope to its
antipode and the opposite path form an isometric $2n$-cycle.  This
will provide a comparison depending only on $n$; the case $n=1$
will be handled separately.

\begin{theorem}\label{thm:orientable-bound}
Let $M$ be a simple orientable matroid on $n>0$ elements.  For
every real $q$ with $-1<q<1$,
\begin{equation}\label{eq:orientable-bound}
  \Mag(M;q)\ge
  \frac{2n(1-q)}{(1+q)(1-q^n)}>1.
\end{equation}
The bound is sharp for every $n$: equality holds for $U_{1,1}$
when $n=1$, and for $U_{2,n}$ when $n\ge2$.
\end{theorem}

\begin{proof}
Choose an orientation $\OM$.  The similarity matrix of the
Hamming cube on $E$ is
\[
  H_E(q)=\begin{pmatrix}1&q\\q&1\end{pmatrix}^{\otimes n}.
\]
Each factor has eigenvalues $1+q$ and $1-q$.  Thus $H_E(q)$ is
positive definite for $-1<q<1$.  Because $G(\OM)$ is isometrically
embedded in the cube, its similarity matrix is a principal
submatrix of $H_E(q)$ and is positive definite as well.

When $n=1$, the magnitude is $2/(1+q)$ and the claim follows.
Assume $n\ge2$, choose a tope $C_0$, and take a shortest path
\[
  C_0,C_1,\ldots,C_n=-C_0.
\]
Such a path has length $n$ and changes each coordinate exactly
once.  Concatenate it with its antipodal image to obtain
\[
  C_0,C_1,\ldots,C_n=-C_0,-C_1,\ldots,-C_{n-1},-C_n=C_0.
\]
If the first path changes coordinates in the order
$e_1,\ldots,e_n$, the concatenation changes them in the order
$e_1,\ldots,e_n,e_1,\ldots,e_n$.  Every shorter arc between two
positions has length at most $n$ and changes distinct
coordinates.  Its length therefore equals the Hamming distance
between its endpoints.  In particular, the $2n$ positions before
the repeated endpoint are distinct and form an isometric copy
of $C_{2n}$.

Apply Lemma~\ref{lem:variational-magnitude} with $b=\boldsymbol1$
and $W$ the vertices of this cycle.  Together with
Theorem~\ref{thm:zaslavsky}, it gives
\[
  \Mag(M;q)\ge\Mag(C_{2n};q).
\]
The similarity matrix of the cycle has constant row sum
\[
  s_n(q)=1+2\sum_{j=1}^{n-1}q^j+q^n
        =\frac{(1+q)(1-q^n)}{1-q}.
\]
Its weighting is the constant vector $s_n(q)^{-1}\boldsymbol1$,
so its magnitude is $2n/s_n(q)$.  Positive definiteness implies
$s_n(q)>0$, while $q^j<1$ for every $j>0$ implies $s_n(q)<2n$.
This proves \eqref{eq:orientable-bound}.  Finally, the tope graph
of $n$ distinct lines through the origin is $C_{2n}$ and the
underlying matroid is $U_{2,n}$, proving sharpness.
\end{proof}

\begin{corollary}\label{cor:no-real-zeros}
If $M$ is a nonempty simple orientable matroid, the numerator of
its reduced magnitude has no real root.  Moreover, $q=-1$ is a
pole of $\Mag(M;q)$.
\end{corollary}

\begin{proof}
Theorem~\ref{thm:orientable-bound} excludes zeros on $(-1,1)$,
and symmetry excludes them for $|q|>1$.  The value at $q=1$
is $1$.  As $q\to-1$ from above, the lower bound in
\eqref{eq:orientable-bound} tends to $+\infty$.  Since magnitude
is rational, $-1$ must therefore be a pole.  These observations
exhaust the real line.
\end{proof}

\subsection{A negative value for a nonorientable matroid}

The proof of Theorem~\ref{thm:orientable-bound} uses orientability
essentially.  We now give an example showing that magnitude need not
be positive for a nonorientable matroid.  Binary projective geometries
are useful here because restrictions to flats and simplified
contractions depend only on rank.  Their flat recursion therefore
reduces to one value for each rank.

\begin{theorem}\label{thm:negative-projective}
The simple connected matroid $P_9=\operatorname{PG}(8,2)$ has
rank $9$ and $511$ elements, and
\[
  \Mag(P_9;9/10)<0.
\]
In particular, its magnitude has at least two real zeros in
$(0,1)$.
\end{theorem}

\begin{proof}
For $r\ge1$, let $P_r$ be the vector matroid on
$\mathbb F_2^r\setminus\{0\}$, and let $P_0$ be empty.  Put
$n_r=2^r-1$ and
\[
  R_r=\prod_{j=0}^{r-1}(1+2^j).
\]
The matroid $P_r$ is simple.  Recall that a matroid is connected
if and only if any two distinct elements lie in a common circuit.
For $r\ge2$, the circuit $\{u,v,u+v\}$ contains any chosen pair
$u,v$, so $P_r$ is connected.  Rank-$k$ flats are the sets of
nonzero vectors of $k$-dimensional subspaces; their number is the
Gaussian binomial coefficient
\[
  \genfrac{[}{]}{0pt}{}{r}{k}_2=\prod_{j=0}^{k-1}
                     \frac{2^{r-j}-1}{2^{k-j}-1}.
\]
Each such flat $A$ has size $n_k$ and restriction $P_k$.
The contraction $P_r/A$ is represented in the quotient vector
space, where each nonzero vector occurs $2^k$ times.  Its
simplification is $P_{r-k}$.

The characteristic polynomial of $P_m$ is
\begin{equation}\label{eq:projective-characteristic}
  \chi_{P_m}(x)=\prod_{j=0}^{m-1}(x-2^j).
\end{equation}
One way to verify this formula is to use the M\"obius function
of the subspace lattice, whose value on a $k$-dimensional
interval is $(-1)^k2^{\binom{k}{2}}$.  This value follows by
induction from the M\"obius recursion and the Gaussian binomial
identity.  Expanding the product in
\eqref{eq:projective-characteristic} gives exactly
\[
  \sum_{k=0}^m \genfrac{[}{]}{0pt}{}{m}{k}_2
        (-1)^k2^{\binom{k}{2}}x^{m-k}.
\]
The lattice of flats, and hence the characteristic polynomial,
is unchanged by simplification.  Thus $R(P_r/A)=R_{r-k}$.

Set $a_r(q)=\Mag(P_r;q)$.  Grouping flats by rank in
\eqref{eq:matroid-flat-recursion} gives
\begin{equation}\label{eq:projective-recursion}
  a_0(q)=1,\qquad
  a_r(q)=\frac{\displaystyle\sum_{k=0}^{r-1}
          (-1)^k\genfrac{[}{]}{0pt}{}{r}{k}_2R_{r-k}q^{n_k}a_k(q)}
               {1-(-1)^rq^{n_r}}.
\end{equation}
Evaluating \eqref{eq:projective-recursion} at $q=9/10$ successively
for $r=1,\ldots,9$ gives
\[
  -228232<a_9(9/10)<-228231.
\]
By Theorem~\ref{thm:magnitude-properties}, $a_9$ is regular on
$[0,1]$, has positive value $R_9$ at zero, and has value one at
$1$.  The intermediate value theorem produces a zero in each
of $(0,9/10)$ and $(9/10,1)$.
Finally, the negative value and simplicity contradict
Theorem~\ref{thm:orientable-bound} if $P_9$ admits an orientation.
Thus this calculation also proves that $P_9$ is nonorientable.
\end{proof}

\section{Orders of poles and signs of coefficients}\label{sec:poles}

We first determine the pole at $q=-1$ from the lattice of
flats, and identify its order with the degree of a natural function on
the topes.  We then construct a real arrangement whose magnitude at
$q=-t$ has a pole at $t=i$ of higher order than its pole at $t=1$.
This forces infinitely many negative coefficients and disproves
Koizumi--Liu's eventual sign alternation conjecture
\cite[Conjecture~7.1(3)]{KL}.

It is convenient to change signs in the variable and write
\[
 \Mag(M;q)=\sum_{\ell\geq0}c_\ell(M)q^\ell,
 \qquad
 A_M(t)=\Mag(M;-t)=\sum_{\ell\geq0}a_\ell(M)t^\ell.
\]
Thus $a_\ell(M)=(-1)^\ell c_\ell(M)$.  The eventual sign
alternation condition that we disprove is
$a_\ell(M)\geq0$ for every sufficiently large degree $\ell$.
Throughout this section, $d(M)$ denotes the
order of the pole of $A_M$ at $t=1$, with $d(M)=0$ if $A_M$ is regular
there.  Equivalently, this is the order of the pole of $\Mag(M;q)$ at
$q=-1$.  The point $q=1$ of the original magnitude is regular.
For a loopless matroid, if
$\nu=\operatorname{ord}_{t=1}A_M(t)$ is the signed order of
vanishing, then $d(M)=\max\{0,-\nu\}$.  Thus a zero has positive
signed order but pole order zero.

\subsection{The pole at \texorpdfstring{$t=1$}{t=1}}

For a loopless matroid $N$, let $\lambda(N)$ be the largest integer $p$
for which there is a chain of flats
\begin{equation}\label{eq:odd-flat-chain}
 \varnothing=F_0\subsetneq F_1\subsetneq\cdots\subsetneq F_p,
 \qquad r(F_j)=j,
 \qquad |F_j\setminus F_{j-1}|\ \text{odd for all }j.
\end{equation}
The chain may stop before the ground set; the empty chain is allowed.
The definition also applies to loopless matroids with parallel elements,
as will be needed for contractions.
For a simple orientable matroid, the proof below constructs topes
on which the signs in the blocks $F_j\setminus F_{j-1}$ can be reversed
independently.  A block of odd size $m_j$ contributes a factor
$1-t^{m_j}$ to the denominator of the magnitude of this subset.
Each such block contributes one to the pole order of this subset's magnitude at $t=1$.

\begin{theorem}\label{thm:pole-order}
Let $M$ be a simple orientable matroid, and put $p=\lambda(M)$.  Then
\begin{equation}\label{eq:pole-order-leading}
 d(M)=p,\qquad
 C_M:=\lim_{t\to1}(1-t)^p A_M(t)>0.
\end{equation}
\end{theorem}

\begin{lemma}\label{lem:pole-parity}
For every loopless matroid $N$,
\[
 \lambda(N)\equiv\operatorname{ord}_{t=1}A_N(t)
 \equiv |E(N)|\pmod2,
\]
where a pole has negative order.
\end{lemma}

\begin{proof}
Let $H$ end a chain of maximum length in \eqref{eq:odd-flat-chain}.
Every parallel class of $N/H$ has even cardinality, since an odd class
would extend the chain.  Hence $|E(N)\setminus H|$ is even, whereas
$|H|\equiv\lambda(N)\pmod2$.
Put $\nu=\operatorname{ord}_{t=1}A_N(t)$ and write
\[
 A_N(t)=c(1-t)^\nu(1+O(1-t)),\qquad c\ne0.
\]
The symmetry identity
\[
 A_N(t^{-1})=(-1)^{|E(N)|}t^{|E(N)|}A_N(t),
\]
which follows from \eqref{eq:magnitude-reciprocity} with $q=-t$,
and $1-t^{-1}=-(1-t)/t$ give
$(-1)^\nu c=(-1)^{|E(N)|}c$.
Thus $\nu\equiv |E(N)|\pmod2$.
\end{proof}

\begin{proof}[Proof of Theorem~\ref{thm:pole-order}]
We first prove $d(N)\leq\lambda(N)$ for every loopless matroid $N$
by induction on $n=|E(N)|$.  The empty matroid is immediate.
For $n>0$, put $k=\lambda(N)$.  The flat recursion
\eqref{eq:matroid-flat-recursion} gives
\begin{equation}\label{eq:signed-flat-recursion}
 \bigl(1-(-1)^{r(N)+n}t^n\bigr)A_N(t)
 =\sum_{\substack{F\in\cF(N)\\F\subsetneq E(N)}}
   (-1)^{r(F)+|F|}R(N/F)t^{|F|}A_{N|F}(t).
\end{equation}
For every proper flat $F$, a chain of flats of $N|F$ is also a chain
of flats of $N$, so $\lambda(N|F)\leq k$.
Induction bounds the poles on the right by $k$, and the factor on
the left has at most a simple zero at $t=1$.  Thus $d(N)\leq k+1$.
If $d(N)=k+1$, then $d(N)>0$ and
$\operatorname{ord}_{t=1}A_N=-d(N)$, contradicting
Lemma~\ref{lem:pole-parity}.  This proves the upper bound.

Now choose an orientation $\OM$ of $M$ and a chain of length
$p=\lambda(M)$ in \eqref{eq:odd-flat-chain}.  The empty matroid is
immediate, so assume $M$ is nonempty and hence $p\geq1$.
Put $\Delta_j=F_j\setminus F_{j-1}$ and $m_j=|\Delta_j|$.
Choose covectors $X_j$ with $Z(X_j)=F_j$ for $0\leq j\leq p$.
Identifying signs with $0,1,-1$, for each
$\sigma\in\{1,-1\}^p$ define
\begin{equation}\label{eq:flag-topes}
 T_\sigma=X_p\circ(\sigma_pX_{p-1})\circ\cdots\circ(\sigma_1X_0).
\end{equation}
Closure under composition and sign reversal makes $T_\sigma$ a
covector.  It has no zero coordinates, so it is a tope.
Its signs outside $F_p$ are fixed, and on $\Delta_j$ they are those
of $\sigma_jX_{j-1}$.  Thus these are $2^p$ distinct topes and
\[
 d(T_\sigma,T_\tau)=\sum_{j:\sigma_j\ne\tau_j}m_j.
\]
Let $W=\{T_\sigma:\sigma\in\{1,-1\}^p\}$.

For $0<t<1$, the matrix $K=K_{G(\OM)}(-t)$ is positive definite:
as in the proof of Theorem~\ref{thm:orientable-bound}, it is a principal
submatrix of the positive definite similarity matrix of the Hamming cube.
Since every $m_j$ is odd, each row of $K_W$ has sum
\[
 \prod_{j=1}^p\bigl(1+(-t)^{m_j}\bigr)
 =\prod_{j=1}^p(1-t^{m_j}).
\]
Theorem~\ref{thm:zaslavsky} and Lemma~\ref{lem:variational-magnitude},
applied with $b=\one$, therefore give
\[
 A_M(t)=\one^{\mathsf T}K^{-1}\one
 \geq\one_W^{\mathsf T}K_W^{-1}\one_W
 =\frac{2^p}{\prod_{j=1}^p(1-t^{m_j})}.
\]
The upper bound $d(M)\leq p$ ensures that the limit defining $C_M$
exists.  Multiplying this inequality by $(1-t)^p$ and letting
$t\to1^-$ yields
\[
 C_M\geq\frac{2^p}{m_1\cdots m_p}>0.
\]
Consequently $d(M)=p$, as required.
\end{proof}

\subsection{The degree of the parity function}

Choose an orientation $\OM$ of $M$, and identify its tope signs with
$\{1,-1\}$.  The function
\[
 \eps(T)=\prod_{e\in E}T_e
\]
records the bipartition of the tope graph, up to an overall sign.
Use real coefficients for the Varchenko--Gelfand filtration in this
subsection, and denote its terms by
\[
 P^{\mathbb R}_{\leq k}
 =\operatorname{span}_{\mathbb R}
   \left\{T\longmapsto\prod_{e\in S}T_e:|S|\leq k\right\}.
\]
This agrees with the filtration by Heaviside degree because
$T_e=2h_e(T)-1$.  Although the displayed expression for $\eps$ has
degree $|E|$, relations among the sign functions can substantially
reduce its degree on the tope set.
There is also a direct reason for this function to appear in the pole
problem: since $(-1)^{d(B,C)}=\eps(B)\eps(C)$, conjugating the similarity
matrix by the diagonal matrix of $\eps$ gives
\[
 A_M(t)=\eps^{\mathsf T}K_{G(\OM)}(t)^{-1}\eps.
\]
Thus the pole records the degeneration of the inverse similarity form
in the parity direction.  The next result expresses its order in terms
of polynomial degree.

\begin{proposition}\label{prop:parity-degree}
For every orientation $\OM$ of a simple matroid $M$,
\[
 \min\{k:\eps\in P^{\mathbb R}_{\leq k}\}=\lambda(M)=d(M).
\]
In particular, this minimum degree is independent of the orientation.
\end{proposition}

\begin{proof}
Equip functions on the topes and formal real linear combinations of
topes with the evaluation pairing
\[
 \left\langle f,\sum_T a_T[T]\right\rangle=\sum_Ta_Tf(T).
\]
For $p\geq1$, define the full annihilator
\[
 Q_p=(P^{\mathbb R}_{\leq p-1})^\perp
 =\{z\in\mathbb R[\cT(\OM)]:
       \langle f,z\rangle=0\text{ for all }f\in P^{\mathbb R}_{\leq p-1}\}.
\]
For a complete flag
$\mathscr F:\varnothing=F_0\subsetneq\cdots\subsetneq F_r=E$,
put $\Delta_j=F_j\setminus F_{j-1}$.
There is a tope $T$ such that independently reversing the signs on
any of the blocks $\Delta_j$ again gives a tope, by the construction
in \eqref{eq:flag-topes}.  For $1\leq p\leq r$, if $T^J$ reverses
the blocks indexed by $J$, form the alternating sum
\begin{equation}\label{eq:prefix-chain}
 \gamma_{\mathscr F,T,p}
 =\sum_{J\subseteq\{1,\ldots,p\}}(-1)^{|J|}[T^J].
\end{equation}
A monomial of degree less than $p$ omits at least one of the first
$p$ blocks.  Its evaluations cancel in pairs under reversal of
that block, so \eqref{eq:prefix-chain} belongs to $Q_p$.
The converse spanning statement is
\cite[Proposition~3.9]{SY}, after extending scalars from $\Z$ to
$\mathbb R$: as the flags and compatible topes vary, these elements
span all of $Q_p$.  The Varchenko--Gelfand filtration is exhausted in
degree $r$ (Section~\ref{sec:preliminaries}; see also
\cite[Definition~3.3]{SY}), so $Q_{r+1}=0$.
Evaluation gives
\[
 \langle\eps,\gamma_{\mathscr F,T,p}\rangle
 =\eps(T)\prod_{j=1}^p\bigl(1-(-1)^{|\Delta_j|}\bigr).
\]
This is nonzero precisely when the first $p$ increments have odd
cardinality.  Thus
$\eps$ is nonzero on $Q_p$ precisely for $1\leq p\leq\lambda(M)$.
By the nondegeneracy of the evaluation pairing,
$\eps\in P^{\mathbb R}_{\leq k}$ if and only if it vanishes on
$Q_{k+1}$.  Its minimum degree is therefore $\lambda(M)$.
Theorem~\ref{thm:pole-order} proves the remaining equality.
\end{proof}

In Section~\ref{sec:filtration}, we use the Varchenko--Gelfand filtration
to recover the motivic zeta function from magnitude cohomology.

\begin{example}\label{ex:uniform-pole}
For the simple uniform matroid $U_{r,n}$, the first $r-1$ increments
of a complete flag are singletons, and the last has cardinality
$n-r+1$.  Hence
\[
 d(U_{r,n})=
 \begin{cases}
 r,&n-r\text{ is even},\\
 r-1,&n-r\text{ is odd}.
 \end{cases}
\]
For example, $d(U_{3,5})=3$, whereas $d(U_{3,4})=2$.
Multiplicativity of magnitude also gives
$\lambda(M_1\oplus M_2)=\lambda(M_1)+\lambda(M_2)$ and
$C_{M_1\oplus M_2}=C_{M_1}C_{M_2}$ for simple orientable matroids.
\end{example}

\subsection{Failure of eventual sign alternation}

We now examine the relation between the poles of magnitude and the
signs of its coefficients.
Let $A(t)=\sum_{\ell\geq0}a_\ell t^\ell$ be a real rational function
with no poles in $|t|<1$.  If $a_\ell\geq0$ for all sufficiently large
$\ell$, then for $|\zeta|=1$ and $0<s<1$,
\begin{equation}\label{eq:eventual-nonnegative-bound}
 |A(s\zeta)|\leq\sum_{\ell\geq0}|a_\ell|s^\ell=A(s)+O(1)
 \qquad(s\to1^-),
\end{equation}
since only finitely many coefficients are negative.
Comparing pole orders shows that every pole on the unit circle has
order at most that at $t=1$, with order zero if $A$ is regular there.
Thus a pole of higher order elsewhere on the unit circle obstructs
eventual nonnegativity.

\begin{theorem}\label{thm:eventual-alternation-counterexample}
There exists a connected simple real-representable matroid $M$ of
rank six on twelve elements such that $a_\ell(M)<0$ for infinitely
many $\ell$.
\end{theorem}

\begin{proof}
The construction starts with a rank-two flat $L$ whose magnitude
has a pole at $t=i$.  We add four pairs of vectors so that the
flats containing $L$ increase this pole order one step at a time.
At the same time, every rank-five flat will have even cardinality.
This prevents an odd-increment chain from reaching rank five and
bounds the pole order at $t=1$ by four.

Let $e_1,\ldots,e_6$ be the standard basis of $\mathbb R^6$, and set
\[
 \begin{gathered}
 u_1=e_1,\quad u_2=e_2,\quad
 u_3=e_1+e_2,\quad u_4=e_1-e_2,\\
 a_j=e_{j+2},\quad b_j=a_j+u_j\quad(1\leq j\leq4).
 \end{gathered}
\]
Let $M$ be the matroid represented by these twelve vectors, and put
\[
 L=\{u_1,u_2,u_3,u_4\},\qquad
 E=L\cup\bigcup_{j=1}^4\{a_j,b_j\}.
\]
The vectors $u_1,u_2,a_1,\ldots,a_4$ form a basis and no two vectors
are proportional, so $M$ is simple of rank six.  Every triple in $L$
and each $\{u_j,a_j,b_j\}$ is a circuit.  Elements in a common circuit
belong to the same connected component.  These circuits join all
elements to the same connected component, so $M$ is connected.

The flats containing $L$ are exactly
\begin{equation}\label{eq:counterexample-flats}
 F_J=L\cup\bigcup_{j\in J}\{a_j,b_j\},
 \qquad J\subseteq\{1,2,3,4\},
\end{equation}
with $r(F_J)=|J|+2$ and $|F_J|=2|J|+4$.
The full classification needed for the two pole estimates is
summarized below, where $B_k=U_{1,1}^{\oplus k}$ and a triangle is
a three-element circuit.
\begin{center}\small
\begin{tabular}{@{}llll@{}}
\toprule
Flat & Restriction & Rank & Cardinality \\
\midrule
$F_J\supseteq L$ & $M|F_J$ & $|J|+2$ & $2|J|+4$ \\
Independent, $L\nsubseteq F$
 & $B_k$, $0\leq k\leq4$ & $k$ & $k$ \\
Contains a triangle, $L\nsubseteq F$
 & $U_{2,3}\oplus B_k$, $0\leq k\leq3$ & $k+2$ & $k+3$ \\
\bottomrule
\end{tabular}
\end{center}
The first row will produce the pole at $i$; the other two rows
give restrictions that are regular there.

To prove the classification, first consider a flat containing $L$.
Modulo $\operatorname{span}(L)$ the pairs have independent
images, and a flat containing $L$ contains either both members of a
pair or neither.
Now let $F$ be a flat not containing $L$.
It contains at most one $u_j$, since any two span $L$.
Because $\{u_j,a_j,b_j\}$ is a circuit and $F$ is a flat,
containing two of these elements implies containing all three.
Thus $F$ contains at most one complete pair $\{a_j,b_j\}$.
If it contains no complete pair, then $u_j\in F$ excludes both
$a_j$ and $b_j$, so $F$ is independent of size at most four.
Otherwise $F$ consists of one triangle $\{u_j,a_j,b_j\}$ and at
most one element from each other pair.  The independent coordinates
$e_3,\ldots,e_6$ show that these remaining elements are coloops
of $M|F$, proving the claim.
In particular, every rank-five flat has size ten or six, and hence
even cardinality.  A chain in \eqref{eq:odd-flat-chain} of length at
least five would have an odd-cardinality rank-five member, so
Theorem~\ref{thm:pole-order} gives $d(M)=\lambda(M)\leq4$.

We show that $A_M$ has a pole of order five at $i$.
The formulas
\[
 A_{B_k}(t)=\frac{2^k}{(1-t)^k},\qquad
 A_{U_{2,3}}(t)=\frac6{(1-t)(1-t+t^2)}
\]
show that all restrictions to flats not containing $L$ are regular
at $i$.  For $m=|J|$, we prove that the limit
\begin{equation}\label{eq:counterexample-leading-limit}
 L_J=\lim_{s\to1^-}(1-s)^{m+1}A_{M|F_J}(is)
\end{equation}
exists and is a positive multiple of $i$.
Hence $A_{M|F_J}$ has a pole of order $m+1$ at $i$.
For $J=\varnothing$, the restriction is $U_{2,4}$, and
\[
 A_{U_{2,4}}(t)=\frac{8(1+t)}{(1-t)(1-t^4)}
\]
gives $L_\varnothing=2i$.
For $m>0$, apply \eqref{eq:signed-flat-recursion} to $M|F_J$ at
$t=is$.  The factor on the left is $1-s^{2m+4}$, with a simple zero
at $s=1$.  All proper-flat terms have poles of order at most $m$
by induction, so $A_{M|F_J}$ has a pole of order at most $m+1$.
The pole bound therefore ensures that the limit $L_J$ in
\eqref{eq:counterexample-leading-limit} exists.
Only $F_{J\setminus\{j\}}$, $j\in J$, can contribute at order $m$
on the right.  Their contractions have rank one, hence $R=2$, and
for every $K\subseteq J$,
\[
 (-1)^{r(F_K)+|F_K|}(is)^{|F_K|}=s^{|F_K|}.
\]
Multiplying the recursion by $(1-s)^m$ and taking the limit gives
\begin{equation}\label{eq:counterexample-leading-recursion}
 (2m+4)L_J=2\sum_{j\in J}L_{J\setminus\{j\}},
 \qquad L_\varnothing=2i.
\end{equation}
All coefficients in this recursion are positive real numbers.
Induction therefore gives $L_J=i\beta_J$ with $\beta_J>0$:
the leading terms have the same phase and cannot cancel.
Taking $J=\{1,2,3,4\}$ gives a pole of
order five at $i$, whereas the pole at $1$ has order at most four.
Theorem~\ref{thm:magnitude-properties} ensures that $A_M$ has no poles
in the open unit disk, so \eqref{eq:eventual-nonnegative-bound}
excludes eventual nonnegativity of its coefficients.
\end{proof}

\begin{corollary}\label{cor:KL-counterexample}
The coefficients of the magnitude of a real central hyperplane
arrangement need not eventually alternate.  In particular,
\cite[Conjecture~7.1(3)]{KL} is false.
\end{corollary}

\begin{proof}
Take the twelve representing vectors in
Theorem~\ref{thm:eventual-alternation-counterexample} as normals to
hyperplanes in $\mathbb R^6$.  They define a real central arrangement
$\mathcal A$ with underlying matroid $M$.  Its magnitude equals
$\Mag(M;q)$ by Theorem~\ref{thm:zaslavsky}.  The
magnitude Euler characteristic identity \eqref{eq:magnitude-euler}
identifies its coefficient of $q^\ell$ with
\[
 \chi_\ell(\mathcal A)
 =\sum_k(-1)^k\operatorname{rank}\MH_{k,\ell}(\mathcal A).
\]
Consequently $(-1)^\ell\chi_\ell(\mathcal A)=a_\ell(M)$, which is
negative for infinitely many $\ell$ by
Theorem~\ref{thm:eventual-alternation-counterexample}.
\end{proof}

\section{Magnitude homology and cohomology of oriented matroids}
\label{sec:homology}

Throughout this section, $\OM$ is a finite simple oriented matroid, $M$ is
its underlying matroid, and $G$ is its tope graph.  Homology has integral
coefficients until the discussion of cup products.  For a tope $C$ and a
subset $F\subseteq E$, write $C\setminus F$ for the sign vector obtained
by replacing the coordinates in $F$ by zero.  Write $\eps_F C$ for the
sign vector obtained by reversing the signs in $F$, and set
\[
 \eps_w C=\eps_{\{e:w_e\text{ is odd}\}}C.
\]
A chain of crossing vector $w$ and terminal tope $C$ necessarily starts
at $\eps_w C$.  In particular, its initial tope is determined by these
two data.

\subsection{The homology calculation}

A \emph{face flag incident to $C$} is a possibly empty weakly decreasing
sequence $\mathcal X=(X_1\geq\cdots\geq X_m)$ of covectors strictly below
$C$.  Repeating a face records repeated crossings of the coordinates
in its zero set.  Its degree and crossing vector are
\[
 k(\mathcal X)=\sum_{i=1}^mr_M(Z(X_i)),\qquad
 w(\mathcal X)=\sum_{i=1}^m\mathbf1_{Z(X_i)}.
\]
Let $\operatorname{Flag}_{k,w}(\OM)_C$ denote the set of such flags with
the indicated data.  Recall that
\[
 U_j(w)=\{e:w_e\geq j\},\qquad
 \kappa(w)=\wt_M(w)=\sum_{j\geq1}r_M(U_j(w)).
\]
We call $(w,C)$ \emph{admissible} if $C\setminus U_j(w)$ is a covector
for every nonempty upper level set $U_j(w)$.  The condition is vacuous
for $w=0$.
For $w\ne0$, the only possible incident face flag of crossing vector
$w$ is
\[
 (C\setminus U_m(w)\geq\cdots\geq C\setminus U_1(w)),
 \qquad m=\max_e w_e.
\]
Indeed, the nested zero sets are recovered from their multiplicities
as the upper level sets in reverse order.  This flag exists exactly
when $(w,C)$ is admissible, and its degree is $\kappa(w)$.
For $w=0$ the unique flag is empty.  These observations explain the
equivalence of the two formulations below.  Admissibility also implies
$w\in\mathcal W(M)$, because covector zero sets are flats; membership
in $\mathcal W(M)$ alone does not ensure incidence to a prescribed $C$.

\begin{theorem}\label{thm:homology}
For every tope $C$ and crossing vector $w$, there is an isomorphism of
abelian groups
\[
 \MH_{k,w}(G;\mathbb Z)_{\bullet\to C}
 \cong\mathbb Z[\operatorname{Flag}_{k,w}(\OM)_C].
\]
Equivalently,
\begin{equation}\label{eq:coordinate-homology}
 \MH_{k,w}(G;\mathbb Z)_{\bullet\to C}\cong
 \begin{cases}
  \mathbb Z,&(w,C)\text{ is admissible and }k=\kappa(w),\\
  0,&\text{otherwise}.
 \end{cases}
\end{equation}
Consequently all integral magnitude homology groups of $G$ are free
abelian.
\end{theorem}

For realizable oriented matroids, the face-flag formula is
\cite[Theorem~6.4]{Koizumi}; the refinement to summands of rank one
is due to Liu \cite[Section~4.2]{Liu}.
The proof uses localization and periodicity.  Localization restricts
chains of crossing vector $w$ and terminal tope $C$ to
$S=\supp(w)$ when $C\setminus S$ is a covector, and asserts that
their homology vanishes otherwise
(Proposition~\ref{prop:homology-localization}).
Periodicity subtracts $\mathbf1_E$ from a crossing vector
$w\geq\mathbf1_E$ and shifts homological degree by $r(M)$
(Proposition~\ref{prop:homology-periodicity}).
The proof of localization in \cite[Theorem~3.1]{Koizumi} uses a
geometric vanishing theorem.  We replace that step by the
tope-interval theorem and finite-poset arguments; the proof of
periodicity in \cite[Theorem~3.2]{Koizumi} then applies without a
realization.

For distinct topes $A,B$, the \emph{metric interval} $[A,B]_G$ is
defined by
\[
 [A,B]_G=\{D:d(A,D)+d(D,B)=d(A,B)\}.
\]
Order it by inclusion of separation sets from $A$, and let $(A,B)_G$
be the poset obtained by deleting its endpoints.  The
interval is \emph{facial} if it is the set
$\cT_X=\{D:X\leq D\}$ of topes above a covector $X$.
Write $\Delta P$ for the order complex of a poset $P$.

\begin{theorem}[{\cite[Theorem~4.4.2]{BLSWZ}}]\label{thm:tope-interval}
If $[A,B]_G=\cT_X$ is facial, then
\begin{equation}\label{eq:interval-sphere}
 \Delta(A,B)_G\simeq S^{r_M(Z(X))-2}.
\end{equation}
Here $S^{-1}=\varnothing$, with reduced homology
$\widetilde H_{-1}(S^{-1};\mathbb Z)=\mathbb Z$.
Otherwise $\Delta(A,B)_G$ is contractible.
\end{theorem}

Fix a tope $A$, and encode all topes relative to $A$ by
\[
 \mathcal C_A=
 \{\mathbf1_{S(A,D)}:D\in\cT(\OM)\}\subseteq\{0,1\}^E.
\]
Both $0$ and $\mathbf1_E$ belong to $\mathcal C_A$.  If $a\in\mathcal
C_A$ and $D_a$ is its tope, the interval $[A,D_a]_G$ is facial precisely
when $A\setminus\operatorname{supp}(a)$ is a covector.
Moreover, if this interval is facial, then
\begin{equation}\label{eq:facial-meet}
 a\wedge q\in\mathcal C_A\qquad(q\in\mathcal C_A),
\end{equation}
where $\wedge$ is coordinatewise minimum.  The tope realizing this
vector is $(A\setminus\operatorname{supp}(a))\circ D_q$.

\subsection{A vanishing argument without a realization}

In this subsection abbreviate $\mathcal C_A$ to $\mathcal C$.  For a
nonnegative integer vector $\beta$, let $\bar\beta$ be its coordinatewise reduction
modulo two, and put
\[
 T=\{\beta\in\mathbb N^E:\bar\beta\in\mathcal C\},\qquad
 N=\{\beta\in\mathbb N^E:\bar\beta\notin\mathcal C\}.
\]
All these sets have the coordinatewise partial order, and all vectors in
this subsection lie in $\N^E$.  For $\alpha\ne0$,
the order complex
\[
 D_\alpha=\Delta\{\beta:0\leq\beta\leq\alpha\}
\]
is the staircase triangulation of $\prod_e\Delta^{\alpha_e}$, a PL ball
of dimension $|\alpha|$.  A simplex belongs to its boundary exactly
when at least one coordinate omits one of the values
$0,1,\ldots,\alpha_e$.

We use the following form of relative Alexander duality, including
the convention $\widetilde H^{-1}(\varnothing;\Z)=\Z$.

\begin{lemma}[{\cite[Lemma~4.3]{Koizumi}}]\label{lem:homology-ball-duality}
Let $K$ be a finite triangulation of a $d$-dimensional PL ball, with
boundary subcomplex $\partial K$.  Partition its vertices into $V_0$
and $V_1$, and let $K_0,K_1$ be the respective induced subcomplexes.
Then, with integral coefficients,
\[
 H_k(K_0,K_0\cap\partial K)
 \cong\widetilde H^{d-k-1}(K_1).
\]
\end{lemma}

\begin{lemma}\label{lem:parity-model}
For $\alpha\ne0$ there are isomorphisms
\begin{align*}
 \mathrm{MC}_{*,\alpha}(G)_{A\to\bullet}
 &\cong C_*(\Delta T_{\leq\alpha},
              \Delta T_{\leq\alpha}\cap\partial D_\alpha),\\
 \MH_{k,\alpha}(G)_{A\to\bullet}
 &\cong\widetilde H^{|\alpha|-k-1}(\Delta N_{\leq\alpha};\mathbb Z).
\end{align*}
\end{lemma}

\begin{proof}
The chain isomorphism is proved by the same argument as
\cite[Lemma~5.8]{Koizumi}, with the vertex $\mathbf1_E$ retained in $T$.
For a magnitude chain $(A=C_0,\ldots,C_k)$, define its
\emph{cumulative crossing vectors} by
\[
 \lambda_i=\sum_{j=1}^i\mathbf1_{S(C_{j-1},C_j)}
 \qquad(0\leq i\leq k).
\]
The correspondence is
\[
 (A=C_0,\ldots,C_k)\longmapsto
 [0=\lambda_0<\cdots<\lambda_k=\alpha].
\]
The parity of $\lambda_i$ specifies $C_i$.  A simplex not in the boundary
contains the two endpoints and has every successive difference in
$\{0,1\}^E\setminus\{0\}$; it therefore determines a unique magnitude
chain.  Deleting an endpoint gives a boundary simplex.  Deleting an
interior vertex gives a boundary simplex exactly when a coordinate
jumps by two, which is exactly when the corresponding magnitude
differential term vanishes.  This proves the chain isomorphism.

For the second isomorphism, the order complexes
$\Delta T_{\leq\alpha}$ and $\Delta N_{\leq\alpha}$ are the induced
subcomplexes of $D_\alpha$ on the two parts of its vertex partition.
Apply Lemma~\ref{lem:homology-ball-duality} with $d=|\alpha|$;
the dual cohomological degree is therefore $|\alpha|-k-1$.
\end{proof}

Call a finite poset \emph{acyclic} if its order complex has zero reduced
integral homology in all degrees.  In particular, the empty poset is
not acyclic.  We use two elementary facts, recorded in
\cite[Lemmas~4.1--4.2]{Koizumi}: an inclusion $P\subseteq Q$ induces an
isomorphism on reduced homology if $P_{\leq x}$ is acyclic for every
$x\in Q\setminus P$; and an intersection-closed nonempty finite family
of acyclic lower ideals has acyclic union.

\begin{lemma}\label{lem:acyclic-support}
The following statements hold.
\begin{enumerate}
\item If $q\in\mathcal C$ and $\gamma\wedge q\in N$, then
      $N_{\leq\gamma}$ is acyclic.
\item If $S\ne\varnothing$ and $N_{\leq\mathbf1_S}$ is acyclic, then
      $N_{\leq\alpha}$ is acyclic for every $\alpha$ with
      $\operatorname{supp}(\alpha)=S$.
\end{enumerate}
\end{lemma}

\begin{proof}
We adapt the induction argument in \cite[Proposition~5.14]{Koizumi}.

\emph{Part (1), vectors with coordinates zero or one.}
First consider $a\in\mathcal C\setminus\{0\}$ for which $[A,D_a]_G$
is not facial.  Chains of crossing vector $a$ starting at $A$ are
geodesic and end at $D_a$.  Removing their endpoints and multiplying
in degree $k$ by $(-1)^k$ identifies
$\mathrm{MC}_{*,a}(G)_{A\to D_a}$ with the augmented chain complex
of $\Delta(A,D_a)_G$, shifted by two.
Theorem~\ref{thm:tope-interval} and
Lemma~\ref{lem:parity-model} thus imply that $N_{\leq a}$ is acyclic;
here integral cohomological acyclicity is equivalent to homological
acyclicity by the universal coefficient theorem.  By
\eqref{eq:facial-meet}, this applies whenever $a,q\in\mathcal C$ and
$a\wedge q\notin\mathcal C$.

If $\gamma\in N$, the lower interval has a maximum.  If $\gamma\notin N$ and
$\gamma\leq\mathbf1_E$, then $\gamma\in\mathcal C$ and
$\gamma\wedge q\notin\mathcal C$.
Thus $N_{\leq\gamma}$ is acyclic by the preceding
paragraph.  This proves (1) when $\gamma\leq\mathbf1_E$.

\emph{Part (1), general vectors.}
Argue simultaneously for all $q\in\mathcal C$ by induction on
$|\gamma|$.  The case $\gamma\in N$ again has a maximum, and the
zero--one case was proved above.  Thus suppose
$H=\{e:\gamma_e\geq2\}$ is nonempty and $\gamma\notin N$.
For every nonempty $X\subseteq H$,
\[
 (\gamma-\mathbf1_X)\wedge q=\gamma\wedge q\in N.
\]
Each lower ideal $N_{\leq\gamma-\mathbf1_X}$ is acyclic by induction.
The family is nonempty and closed under intersections, since
$N_{\leq\gamma-\mathbf1_X}\cap N_{\leq\gamma-\mathbf1_Y}=N_{\leq\gamma-\mathbf1_{X\cup Y}}$.
The lower-ideal gluing fact therefore shows that
\[
 W=\bigcup_{\varnothing\ne X\subseteq H}N_{\leq\gamma-\mathbf1_X}
\]
is acyclic.
For $\beta\in N_{\leq\gamma}\setminus W$ we have
$\beta_e=\gamma_e$ on $H$, and therefore
\[
 W_{\leq\beta}
 =\bigcup_{\varnothing\ne X\subseteq H}N_{\leq\beta-\mathbf1_X}.
\]
Since $\beta\in N_{\leq \gamma}$, we have $\bar\beta\notin\mathcal C$,
whereas $\bar\gamma\in\mathcal C$. Note that
\begin{equation}\label{eq:parity-meet-induction}
 (\beta-\mathbf1_X)\wedge\bar\gamma=\bar\beta\in N.
\end{equation}
Indeed, if $e\in H$, the left coordinate before taking the minimum is at least one
and $\bar\beta_e=\bar\gamma_e$; otherwise use
$\beta_e\leq\gamma_e\leq1$.  Induction applied with
$q=\bar\gamma$ makes $N_{\leq\beta-\mathbf1_X}$ acyclic for every
nonempty $X\subseteq H$.  The same intersection formula with
$\gamma$ replaced by $\beta$ permits gluing, so $W_{\leq\beta}$
is acyclic.  The inclusion
$W\subseteq N_{\leq\gamma}$ consequently induces an isomorphism on
reduced homology, proving (1).

\emph{Part (2), induction with fixed support.}
Induct on $|\alpha|$ among vectors with support $S$.
The case $\alpha=\mathbf1_S$ is the hypothesis,
and $\alpha\in N$ gives a maximum.  Otherwise take
$H=\{e:\alpha_e\geq2\}$ and form $W$ as above with
$\gamma=\alpha$.  Every $\alpha-\mathbf1_X$ still has support $S$,
so induction and the intersection formula make $W$ acyclic.
For a remaining $\beta$, formula
\eqref{eq:parity-meet-induction}, now with $\gamma=\alpha$, and the
already proved part (1) make all ideals in $W_{\leq\beta}$ acyclic.
Their intersections have the same form, so gluing makes
$W_{\leq\beta}$ acyclic.  The inclusion fact then transfers
acyclicity from $W$ to $N_{\leq\alpha}$.
\end{proof}

\begin{proposition}[Localization]\label{prop:homology-localization}
Let $w\ne0$, $S=\operatorname{supp}(w)$, and let $C$ be a tope.  If
$X=C\setminus S$ is a covector, restriction induces a chain isomorphism
\[
 \mathrm{MC}_{*,w}(G)_{\bullet\to C}
 \cong
 \mathrm{MC}_{*,w|_S}(G(\OM|_S))_{\bullet\to C|_S}.
\]
If $C\setminus S$ is not a covector, then
$\MH_{*,w}(G)_{\bullet\to C}=0$.
\end{proposition}

\begin{proof}
Every tope in such a chain agrees with $C$ outside $S$.  When $X$ is a
covector, these topes are exactly $\cT_X$.  Restriction maps $\cT_X$
bijectively to the topes of $\OM|_S$: lift a tope of the restriction
to a covector $Y$ and use $X\circ Y$ for its inverse image.  Separation
sets are preserved, proving the chain assertion.

If $C\setminus S$ is not a covector, form the posets $T,N$ of
Lemma~\ref{lem:parity-model} with base tope $A=C$.
For $s=\mathbf1_S$, either $s\notin\mathcal C_C$, in which
case $N_{\leq s}$ has a maximum, or the interval $[C,D_s]_G$ is not
facial, in which case the first paragraph of the preceding proof
applies.  Part (2) of Lemma~\ref{lem:acyclic-support} and Alexander
duality now give $\MH_{*,w}(G)_{C\to\bullet}=0$.  Chain reversal,
with the sign $(-1)^{k(k+1)/2}$ in degree $k$, gives the same vanishing
with terminal tope $C$.
\end{proof}

\subsection{Periodicity and face flags}

\begin{proposition}[Periodicity]\label{prop:homology-periodicity}
If $E\ne\varnothing$, $r=r(M)$, and $w\geq\mathbf1_E$, then
\[
 \MH_{k,w}(G)_{\bullet\to C}
 \cong\MH_{k-r,w-\mathbf1_E}(G)_{\bullet\to C}.
\]
Negative homological degrees are interpreted as zero.
\end{proposition}

\begin{proof}
This is proved for real hyperplane arrangements in \cite[Section 5]{Koizumi}.
The argument uses only covector combinatorics and
Theorem~\ref{thm:tope-interval}, so it also applies to oriented matroids.
The same argument also shows that the subcomplex spanned by chains
whose cumulative crossing vectors avoid $\mathbf1_E$ is acyclic
\cite[Section~5.4]{Koizumi}.
\end{proof}

\begin{proof}[Proof of Theorem~\ref{thm:homology}]
We induct on $|w|$, simultaneously for all simple oriented matroids.
For $w=0$, the chain $(C)$ generates $\MH_{0,0}(G)_{\bullet\to C}$,
and the empty flag is the unique face flag of crossing vector zero.
All other degrees vanish.
Suppose $w\ne0$, and set
\[
 S=\supp(w),\qquad X=C\setminus S,\qquad
 r_S=r_M(S),\qquad w'=w|_S-\mathbf1_S.
\]
For every face flag $(X_1\geq\cdots\geq X_m)$ of crossing vector $w$,
\[
 S=\bigcup_{i=1}^m Z(X_i)=Z(X_m),\qquad X_m=X.
\]
If $X$ is not a covector, the flag set is empty and
Proposition~\ref{prop:homology-localization} gives
$\MH_{*,w}(G)_{\bullet\to C}=0$.

If $X$ is a covector, restriction identifies the covector intervals
$[X,C]$ and $[0,C|_S]$ of $\OM$ and $\OM|_S$, respectively,
and preserves zero sets and their ranks.  Restricting a flag and
deleting its last face therefore gives a bijection
\[
 \begin{aligned}
 \Flag_{k,w}(\OM)_C
 &\longrightarrow\Flag_{k-r_S,w'}(\OM|_S)_{C|_S},\\
 (X_1,\ldots,X_m)&\longmapsto
 (X_1|_S,\ldots,X_{m-1}|_S).
 \end{aligned}
\]
On homology, Propositions~\ref{prop:homology-localization}
and~\ref{prop:homology-periodicity} give
\[
 \begin{aligned}
 \MH_{k,w}(G)_{\bullet\to C}
 &\cong\MH_{k,w|_S}(G(\OM|_S))_{\bullet\to C|_S}\\
 &\cong\MH_{k-r_S,w'}(G(\OM|_S))_{\bullet\to C|_S}\\
 &\cong\Z[\Flag_{k-r_S,w'}(\OM|_S)_{C|_S}].
 \end{aligned}
\]
The last isomorphism follows by induction, since $\OM|_S$ is simple
and $|w'|=|w|-|S|<|w|$.  The flag bijection proves the first assertion.

As observed before the theorem, the crossing vector recovers the
nested zero sets as $Z(X_i)=U_{m-i+1}(w)$, where $m=\max_e w_e$.
The resulting flag exists exactly when $(w,C)$ is admissible and
has degree $\kappa(w)$.  This proves \eqref{eq:coordinate-homology}.
\end{proof}

\subsection{Initial oriented matroids and the cohomology basis}
\label{subsec:cohomology-basis}

We next express admissibility in terms of a tope set.  This requires
initial oriented matroids that may have parallel elements, even
though $\OM$ itself is simple.

\begin{lemma}\label{lem:flag-topes}
Let $\mathcal N$ be a loopless oriented matroid and let
$\varnothing=F_0\subsetneq \cdots\subsetneq F_s=E$
be a flag of flats.  Set
\[
 \mathcal N_{\mathcal F}
 =\bigoplus_{i=1}^s(\mathcal N|_{F_i})/F_{i-1}.
\]
Then, as sets of sign vectors,
\[
 \cT(\mathcal N_{\mathcal F})
 =\{D\in\cT(\mathcal N):D\setminus F_i\in\cL(\mathcal N)
                    \text{ for all }i\}.
\]
In particular, these topes are topes of $\mathcal N$.
\end{lemma}

\begin{proof}
For a tope on the right, restrict $D\setminus F_{i-1}$ to $F_i$ to
obtain the tope of the $i$th factor.  Conversely, lift a tope $D_i$ of
that factor to a covector $Y_i$ of $\mathcal N$ which is zero on
$F_{i-1}$ and equals $D_i$ on $F_i\setminus F_{i-1}$.  Such a lift
exists by the definitions of restriction and contraction.  The
composition
\[
 D=Y_s\circ Y_{s-1}\circ\cdots\circ Y_1
\]
has the prescribed sign on every layer, and hence is a tope.  Moreover
$Y_s\circ\cdots\circ Y_{i+1}=D\setminus F_i$; for $i=s$ use the zero
covector.  These are the required incidences.
\end{proof}

Recall that $\mathcal W(M)$ consists of the nonnegative integer weights
$w$ for which $M_w$ is loopless, equivalently those whose nonempty upper
level sets are flats.  For such $w$, arrange the distinct upper level
sets in increasing order and add $\varnothing,E$ as necessary to
obtain a flag $\mathcal F(w)$.  Define
\begin{equation}\label{eq:initial-oriented-matroid}
 \OM_w=\bigoplus_{i=1}^s(\OM|_{F_i})/F_{i-1},\qquad
 \cT_w=\cT(\OM_w)\subseteq\cT(\OM).
\end{equation}
The underlying matroid is $M_w$: a basis maximizes $w$ precisely when
it meets every upper level set in a basis of that set.  This is the
same characterization as the bases of the displayed direct sum.
By Lemma~\ref{lem:flag-topes}, $(w,C)$ is admissible exactly when
$C\in\cT_w$.

From now on cohomology has coefficients in $\Ftwo$.  The universal
coefficient theorem and Theorem~\ref{thm:homology} give a distinguished
basis, since a one-dimensional vector space over $\Ftwo$ has a unique
nonzero element.  Denote that element by
\[
 u_{w,C}\in
 \MH^{\kappa(w),w}(G;\Ftwo)_{\eps_w C\to C}
 \qquad(C\in\cT_w).
\]
Set $H_w=\MH^{\kappa(w),w}(G;\Ftwo)$ for $w\in\mathcal W(M)$ and
$H_w=0$ otherwise.  All other cohomological degrees of crossing vector
$w$ vanish.  We obtain the canonical vector-space isomorphism
\begin{equation}\label{eq:Phi}
 \Phi_w:\operatorname{Fun}(\cT_w,\Ftwo)\xrightarrow{\ \cong\ }H_w,
 \qquad f\longmapsto\sum_{C\in\cT_w}f(C)u_{w,C}.
\end{equation}
At $w=0$ this is the usual identification of degree-zero cohomology
with functions on vertices.  At positive weights it is an
identification of vector spaces; the cup product adds crossing vectors.

\subsection{The cup product}

\begin{theorem}\label{thm:cup-product}
For admissible $(\alpha,B)$ and $(\beta,C)$, put
$\gamma=\alpha+\beta$.  Then
\begin{equation}\label{eq:cup-product}
 u_{\alpha,B}\smile u_{\beta,C}=
 \begin{cases}
 u_{\gamma,C},&B=\eps_\beta C,\quad(\gamma,C)\text{ admissible},\quad
 \kappa(\gamma)=\kappa(\alpha)+\kappa(\beta),\\
 0,&\text{otherwise}.
 \end{cases}
\end{equation}
The unit is $\sum_{C\in\cT(\OM)}u_{0,C}$, and the $u_{0,C}$ are
pairwise orthogonal idempotents.
\end{theorem}

The three conditions have different roles.  The first matches the
terminal tope of the first factor with the initial tope of the second.
The second says that the resulting endpoint summand exists.  The third
matches its cohomological degree with the sum of the degrees of the
factors.  Their necessity follows from the cochain product and
Theorem~\ref{thm:homology}.

\begin{proof}
This is proved for real hyperplane arrangements by Liu \cite[Theorem~7.3]{Liu}.
The proof essentially applies to oriented matroids as well, except for two parts that rely on realizability.
The first is the homology computation used in \cite[Theorem~4.3]{Liu}, which can be replaced by Theorem~\ref{thm:homology}.
The second is the acyclicity of nonspecial chains used in
\cite[Theorem~4.15]{Liu}, supplied by the proof of
Proposition~\ref{prop:homology-periodicity} after localization to the support.
The rest of the argument uses only oriented matroids.
\end{proof}

The following tope compatibility will be used in Section~\ref{sec:filtration}.

\begin{lemma}\label{lem:tope-compatibility}
Let $\alpha,\beta,\gamma\in\mathcal W(M)$, $C\in \cT(\OM)$ and suppose that
$$
\gamma=\alpha+\beta,\qquad \kappa(\gamma)=\kappa(\alpha)+\kappa(\beta).
$$
If $(\gamma,C)$ is admissible, then $(\alpha,\varepsilon_\beta(C))$ and $(\beta,C)$ are admissible.
\end{lemma}

\begin{proof}
Apply \cite[Proposition~3.10 and Lemma~3.12]{Liu} to the nonempty
upper level sets of $\alpha$ and $\beta$, with repetitions.  Their
incidence vectors sum to $\gamma$ and their ranks sum to
$\kappa(\gamma)$.  These results give $D\setminus F\in\cL(\OM)$
and $\eps_FD\in\cT_\gamma$ for every $D\in\cT_\gamma$ and every
such flat $F$.  Composing the reversals for the upper level sets of
$\beta$ gives $\eps_\beta C\in\cT_\gamma$.  Applying the covector
assertion at $C$ and $\eps_\beta C$ proves admissibility.
\end{proof}

\section{The Varchenko--Gelfand filtration and the motivic zeta function}
\label{sec:filtration}

Let $\OM$ be a simple oriented matroid with underlying matroid $M$ of
rank $r$, and put $G=G(\OM)$.
Throughout this section cohomology has coefficients in $\Ftwo$.
The preceding section identifies the nonzero crossing-vector components
of cohomology with functions on the topes of initial oriented matroids.
Applying polynomial degree to those functions extends the classical
Varchenko--Gelfand filtration from length zero to every length.

We use the canonical classes $u_{w,C}$ from
Section~\ref{subsec:cohomology-basis} to define the filtration on
mod-$2$ magnitude cohomology.
We obtain a canonical filtered algebra for each $\OM$, preserved by
relabeling and reorientation; Theorem~\ref{thm:associated-graded}
shows that its associated graded algebra depends only on the labeled
underlying matroid $M$.
In the present paper, the filtration is not constructed for integral
coefficients: an integral cup product formula remains conjectural even in the
realizable setting \cite[Conjecture~7.4]{Liu}.

\subsection{Construction and multiplicativity}
For $w\in\mathcal W(M)$, recall the initial oriented matroid $\OM_w$,
its tope set $\cT_w$, and the canonical isomorphism
\[
 \Phi_w:V(\OM_w)\xrightarrow{\sim}H_w,
 \qquad H_w=\MH^{\kappa(w),w}(G;\Ftwo),
\]
given by $\Phi_w(f)=\sum_{C\in\cT_w}f(C)u_{w,C}$.
Put $H_w=0$ for $w\notin\mathcal W(M)$.
Define
\begin{equation}\label{eq:extended-filtration}
 F_pH_w=\Phi_w(P_pV(\OM_w)),\qquad
 F_p\MH^{k,\ell}(G;\Ftwo)
 =\bigoplus_{\substack{w\in\mathcal W(M)\\|w|=\ell,\ \kappa(w)=k}}
 F_pH_w.
\end{equation}
The initial matroid $M_w$ has rank $r$, so $F_{-1}=0$ and $F_r=\MH$
in every bidegree.
At $w=0$, the classes $u_{0,C}$ are the characteristic functions of
individual topes, and \eqref{eq:extended-filtration} is the ordinary
Varchenko--Gelfand filtration.

\begin{theorem}\label{thm:filtered-algebra}
The filtration \eqref{eq:extended-filtration} is multiplicative:
\[
 F_p\MH^{k,\ell}(G;\Ftwo)\smile
 F_q\MH^{k',\ell'}(G;\Ftwo)
 \subseteq F_{p+q}\MH^{k+k',\ell+\ell'}(G;\Ftwo).
\]
It is preserved by relabeling and reorientation of $\OM$.
\end{theorem}

\begin{proof}
Fix $\alpha,\beta\in\mathcal W(M)$ and put $\gamma=\alpha+\beta$.
The target component vanishes unless
\begin{equation}\label{eq:compatible-weights}
 \gamma\in\mathcal W(M),\qquad
 \kappa(\gamma)=\kappa(\alpha)+\kappa(\beta).
\end{equation}
The inequality $\kappa(\alpha+\beta)\leq
\kappa(\alpha)+\kappa(\beta)$ always holds: for each basis, its
$\alpha$- and $\beta$-weights are bounded by the corresponding maxima.
Equality means that some basis maximizes both weights; in that case
the bases maximizing $\alpha+\beta$ are exactly those maximizing both
$\alpha$ and $\beta$.
Under these conditions Lemma~\ref{lem:tope-compatibility} shows that
$(\alpha,\varepsilon_\beta C)$ and $(\beta,C)$ are admissible for
every $C\in\cT_\gamma$.
The cup product formula of Theorem~\ref{thm:cup-product} therefore reads
\begin{equation}\label{eq:function-cup-product}
 \Phi_\alpha(f)\smile\Phi_\beta(g)
 =\Phi_\gamma\bigl(C\longmapsto f(\eps_\beta C)g(C)\bigr).
\end{equation}
On Heaviside functions,
\[
 h_e(\eps_\beta C)=h_e(C)+(\beta_e\bmod2).
\]
Replacing any variable by itself plus a constant preserves polynomial
degree.
Thus if $f$ and $g$ have Heaviside degree at most $p$ and $q$,
respectively, the product function in \eqref{eq:function-cup-product}
has degree at most $p+q$.
Summing over crossing vectors proves multiplicativity.
Relabeling permutes the Heaviside generators, and reorientation replaces
some of them by $1+h_e$.
Both operations preserve the degree filtration and the canonical
endpoint classes.
\end{proof}

There is also a description using only multiplication by degree-zero
cohomology.
Let $c_w=\sum_{C\in\cT_w}u_{w,C}$.
Since right multiplication by a vertex function evaluates it at the
terminal tope, one has
\begin{equation}\label{eq:intrinsic-filtration}
 F_pH_w=c_w\smile P_p\MH^{0,0}(G;\Ftwo).
\end{equation}
Indeed, every Heaviside polynomial on $\cT_w$ is the restriction of the
same polynomial on $\cT(\OM)$.
Although $H_w$ is identified with a function space, its pointwise
multiplication is not its cup product: cup products add crossing vectors.

\subsection{Recovering the zeta function}
Introduce the series
\begin{equation}\label{eq:three-variable-series}
 \mathscr H_{\OM}(s,t,T)=
 \sum_{p,k,\ell\geq0}
 \dim_{\Ftwo}\gr_p^F\MH^{k,\ell}(G;\Ftwo)\,s^pt^kT^\ell.
\end{equation}
The three variables distinguish polynomial degree, cohomological degree,
and length, in that order.
Its coefficients in $T$ are polynomials in $s,t$.

\begin{theorem}\label{thm:zeta-filtration}
For every simple oriented matroid $\OM$ with underlying matroid $M$,
\begin{equation}\label{eq:zeta-from-filtration}
 Z_M(x,T)=\mathscr H_{\OM}(-x^{-1},x^{-1},T).
\end{equation}
Equivalently, for each $\ell\geq0$,
\[
 [T^\ell]Z_M(x,T)=\sum_{p,k}(-1)^p
 \dim_{\Ftwo}\gr_p^F\MH^{k,\ell}(G;\Ftwo)\,x^{-p-k}.
\]
\end{theorem}

\begin{proof}
Theorem~\ref{thm:homology} and the construction give
\begin{equation}\label{eq:H-weight-expansion}
 \mathscr H_{\OM}(s,t,T)=
 \sum_{w\in\mathcal W(M)}
 \left(\sum_{p=0}^r b_p(M_w)s^p\right)t^{\kappa(w)}T^{|w|}.
\end{equation}
By \eqref{eq:vg-characteristic}, the summand indexed by $w$ becomes
\[
 \sum_{p=0}^r(-1)^p b_p(M_w)x^{-p-\kappa(w)}T^{|w|}
 =\chi_{M_w}(x)x^{-r-\wt_M(w)}T^{|w|}
\]
after the stated substitution.
The terms with $w\notin\mathcal W(M)$ in
\eqref{eq:zeta-definition} vanish, so the two sums agree.
\end{proof}

At $T=0$, this theorem says
\[
 x^{-r}\chi_M(x)=\sum_p(-1)^p
 \dim_{\Ftwo}\gr_p^P V(\OM)\,x^{-p},
\]
the classical characteristic-polynomial formula for the
Varchenko--Gelfand filtration.
At $x=-1$ it gives
\begin{align*}
 Z_M(-1,T)
 &=\mathscr H_{\OM}(1,-1,T)\\
 &=\sum_{k,\ell}(-1)^k\dim_{\Ftwo}\MH^{k,\ell}(G;\Ftwo)T^\ell
 =\Mag(G;T).
\end{align*}
The substitution $s=1$ sums the dimensions of the associated graded
pieces and hence forgets the filtration, while $t=-1$ takes the
Euler characteristic.
This recovers the magnitude-refined Las Vergnas--Zaslavsky theorem from
Section~\ref{sec:zaslavsky} and explains its normalization.

\begin{example}
For the free matroid $U_{r,r}$, every nonnegative weight contributes,
$M_w=M$, and $\kappa(w)=|w|$.
Hence
\begin{equation}\label{eq:free-H}
 \mathscr H_{\OM}(s,t,T)=\frac{(1+s)^r}{(1-tT)^r},\qquad
 Z_{U_{r,r}}(x,T)=\left(\frac{x-1}{x-T}\right)^r.
\end{equation}
The tope graph is the $r$-cube, and $x=-1$ gives its magnitude
$2^r/(1+T)^r$.
\end{example}

\begin{example}
The matroid $U_{2,3}$, whose tope graph is the six-cycle, illustrates
why cohomological degree and length must be kept separate.
Its contributing weights fall into the two families below, extending
Example~\ref{ex:initial-u23}:
\begin{center}
\begin{tabular}{@{}llll@{}}
\toprule
Weight & Initial matroid & Degree $k$ & Length $\ell$ \\
\midrule
$a\one_E$ & $U_{2,3}$ & $2a$ & $3a$ \\
$a\one_E+b\one_{\{e\}}$ & $U_{1,1}\oplus U_{1,2}$ & $2a+b$ & $3a+b$ \\
\bottomrule
\end{tabular}
\end{center}
Here $a\geq0$, $b\geq1$, and $e\in E$.
The graded-dimension polynomials of the two initial matroids are
$1+3s+2s^2$ and $(1+s)^2$, respectively.
Summing over $a$ and $b$ produces the denominators $1-t^2T^3$ and
$1-tT$, and the three choices of $e$ give
\begin{equation}\label{eq:triangle-H}
 \mathscr H_{\OM}(s,t,T)=
 \frac{1+3s+2s^2+3(1+s)^2\dfrac{tT}{1-tT}}{1-t^2T^3}.
\end{equation}
At length three the weights $\one_E$ and $3\one_{\{e\}}$ give,
respectively,
\[
 \sum_p\dim\gr_p^F\MH^{2,3}s^p=1+3s+2s^2,\qquad
 \sum_p\dim\gr_p^F\MH^{3,3}s^p=3(1+s)^2.
\]
Their contributions to the zeta function are
$x^{-4}(x-1)(x-2)$ and $3x^{-5}(x-1)^2$.
The additional power of $x^{-1}$ in the second term is recorded by the
cohomological degree, independently of Heaviside degree.
\end{example}

\subsection{Associated graded algebra}
Passing to the associated graded removes the lower-degree terms
introduced by sign reversal in \eqref{eq:function-cup-product}.
The resulting algebra admits a description in terms of $M$ alone.

\begin{theorem}\label{thm:associated-graded}
There is a canonical isomorphism of spaces graded by polynomial degree,
cohomological degree, and crossing vector
\begin{equation}\label{eq:graded-os-sum}
 \gr_F\MH^{*,*}(G;\Ftwo)
 \cong\bigoplus_{w\in\mathcal W(M)}\OS^\bullet(M_w;\Ftwo),
\end{equation}
where the two asterisks on the left denote cohomological degree and
length, further refined by crossing vector and polynomial degree.
The summand of weight $w$ lies in cohomological degree
$\kappa(w)$ and length $|w|$.
If $\alpha,\beta$ satisfy \eqref{eq:compatible-weights}, there are
homomorphisms of graded algebras
\[
 \rho_{\alpha,\gamma}:\OS^\bullet(M_\alpha;\Ftwo)
 \longrightarrow\OS^\bullet(M_\gamma;\Ftwo),\qquad a_e\longmapsto a_e,
\]
and similarly for $\beta$.
Under \eqref{eq:graded-os-sum}, the product of $a$ in weight $\alpha$
and $b$ in weight $\beta$ is
\begin{equation}\label{eq:graded-product}
 a*b=\rho_{\alpha,\gamma}(a)\rho_{\beta,\gamma}(b)
 \quad\text{in weight }\gamma=\alpha+\beta
\end{equation}
when \eqref{eq:compatible-weights} holds, and is zero otherwise.
In particular, the associated graded algebra is commutative and depends
only on the labeled matroid $M$.
\end{theorem}

\begin{proof}
Apply \eqref{eq:vg-os} to each $\OM_w$ to obtain the vector-space
isomorphism.
Under \eqref{eq:compatible-weights}, restriction along
$\cT_\gamma\subseteq\cT_\alpha$ gives a filtered homomorphism of
pointwise function algebras, preserving each Heaviside generator.
Its associated graded map is $\rho_{\alpha,\gamma}$; this construction
in particular proves that the indicated assignment respects the circuit
relations.
The argument for $\beta$ is identical.
For a Heaviside polynomial $f$ of degree at most $p$, the difference
$f(\eps_\beta C)-f(C)$ has degree at most $p-1$.
Therefore the sign reversal in \eqref{eq:function-cup-product}
disappears on the associated graded, giving \eqref{eq:graded-product}.
Every map is determined by its action on the standard generators and
thus by the relevant initial matroids.
The criterion \eqref{eq:compatible-weights} is symmetric in
$\alpha,\beta$, and the Orlik--Solomon algebras over $\Ftwo$ are
commutative; this proves the last assertion.
\end{proof}

\bibliographystyle{amsplain}
\bibliography{references}
\end{document}